\documentclass[11pt]{amsart}
\usepackage{geometry}     
\usepackage[parfill]{parskip}    
\usepackage{graphicx}
\usepackage{amssymb}
\usepackage{epstopdf}
\usepackage{xcolor}
\usepackage{fullpage}
\usepackage{tikz-cd}
\usepackage{mathrsfs}
\usepackage[all, 2cell]{xy}
\usepackage{hyperref}
\usepackage{marginnote}

\theoremstyle{plain}
\newtheorem{theorem}{Theorem}[section]
\newtheorem*{theorem*}{Theorem}

\newtheorem{lemma}[theorem]{Lemma}
\newtheorem{corollary}[theorem]{Corollary}
\newtheorem{proposition}[theorem]{Proposition}

\newtheorem{remark}[theorem]{Remark}

\theoremstyle{definition}
\newtheorem{example}[theorem]{Example}

\newtheorem{definition}[theorem]{Definition}
\newtheorem{definitions}[theorem]{Definitions}
\newtheorem{notation}[theorem]{Notation}

\newcommand{\A}{\mathcal{A}}
\newcommand{\B}{\mathcal{B}}
\newcommand{\C}{\mathcal{C}}
\newcommand{\D}{\mathrm{D}}

\newcommand{\F}{\mathcal{F}}
\newcommand{\G}{\mathcal{G}}
\newcommand{\I}{\mathscr{I}}
\newcommand{\Ical}{\mathcal{I}}
\newcommand{\J}{\mathcal{J}}
\newcommand{\M}{\mathcal{M}}
\newcommand{\N}{\mathcal{N}}

\renewcommand{\S}{\mathcal{S}}
\newcommand{\T}{\mathcal{T}}
\newcommand{\U}{\mathcal{U}}
\newcommand{\V}{\mathcal{V}}
\newcommand{\X}{\mathcal{X}}
\newcommand{\Y}{\mathcal{Y}}
\newcommand{\Z}{\mathcal{Z}}

\renewcommand{\t}{\mathsf{t}}
\newcommand{\f}{\mathsf{f}}
\renewcommand{\u}{\mathsf{u}}
\renewcommand{\v}{\mathsf{v}}

\newcommand{\fp}[1]{{#1}^{\mathsf{fp}}}
\renewcommand{\c}{\mathrm{c}}
\newcommand{\op}{\mathrm{op}}
\newcommand{\Mod}[1]{\mathrm{Mod}(#1)}
\newcommand{\Modop}[1]{\mathrm{Mod}(#1^{\op})}
\renewcommand{\mod}[1]{\mathrm{mod}(#1)}
\newcommand{\Ab}{\mathrm{Ab}}
\newcommand{\Cogen}[1]{\mathrm{Cogen}(#1)}

\newcommand{\Der}[1]{\mathrm{D}(#1)}
\newcommand{\Db}[1]{\mathrm{D}^\mathrm{b}(\mod{#1})}
\newcommand{\K}[1]{\mathrm{K}^{[0,1]}(\mathrm{Inj}#1)}
\newcommand{\thick}[1]{\mathrm{thick}(#1)}
\newcommand{\Prod}[1]{\mathrm{Prod}(#1)}

\newcommand{\Kb}[1]{\mathrm{K}^{\mathrm{b}}(\mathrm{Inj}#1)}
\newcommand{\Inj}[1]{\mathrm{Inj}(#1)}

\newcommand{\torsf}[1]{\widetilde{\mathbf{F}}(#1)}

\newcommand{\y}{\mathrm{y}}
\newcommand{\ten}{\mathrm{t}}
\newcommand{\Ker}[1]{\mathrm{Ker}(#1)}

\renewcommand{\Im}[1]{\mathrm{Im}(#1)}
\newcommand{\Hom}[1]{\mathrm{Hom}_{#1}}
\newcommand{\Ext}[1]{\mathrm{Ext}_{#1}}
\newcommand{\End}[1]{\mathrm{End}_{#1}}
\newcommand{\Rad}{\mathrm{Rad}}
\newcommand{\Soc}{\mathrm{Soc}}
\newcommand{\tstr}[1]{\mathbb{T}_{#1}}
\newcommand{\ststr}[1]{\mathbb{D}_{#1}}
\newcommand{\heart}[1]{\mathcal{H}_{#1}}
\renewcommand{\H}[1]{\mathrm{H}^0_{#1}}

\newcommand{\Zg}[1]{\mathbf{Zg}(#1)}
\newcommand{\open}[1]{\mathscr{O}(#1)}
\newcommand{\closed}[1]{\mathcal{N}_{#1}}
\newcommand{\ZgInt}[1]{\mathbf{Zg}^{[0,1]}(D(#1))}
\newcommand{\InjInd}[1]{\mathbf{Inj}(#1)}
\newcommand{\Sp}[1]{\mathbf{Sp}(#1)}
\newcommand{\tors}[1]{\mathbf{tors}(#1)}

\newcommand{\Cosilt}[1]{\mathbf{Cosilt}(#1)}
\newcommand{\CosiltZg}[1]{\mathbf{MaxRigid}(#1)}
\newcommand{\Cosiltpair}[1]{\mathbf{CosiltPair}(#1)}
\newcommand{\Ann}[1]{\mathrm{Ann}(#1)}
\newcommand{\brick}[1]{\mathbf{brick}(#1)}

\newcommand{\Tpair}[2]{(\mathcal{#1}, \mathcal{#2 })}

\newcommand{\srigid}{{grain}}
\newcommand{\Ri}{{\mathbf{Grain}_{\rm cmi}}}

\title{Torsion pairs via the Ziegler spectrum}
\author{Lidia Angeleri H\"ugel, Rosanna Laking and Francesco Sentieri}
\address{Dipartimento di Informatica - Settore di Matematica, Universit\`a degli Studi di Verona, Strada le Grazie 15 - Ca' Vignal, I-37134 Verona, Italy} 
\email{lidia.angeleri@univr.it, rosanna.laking@univr.it, francesco.sentieri@univr.it}
\thanks{Acknowledgments: The authors were partially supported by the project \textit{REDCOM: Reducing complexity in algebra, logic, combinatorics} in the programme  Ricerca Scientifica di Eccellenza 2018 of the Fondazione Cariverona. The first and second named authors are members of the network INdAM-G.N.S.A.G.A and acknowledge support from the project  \textit{SQUARE: Structures for Quivers, Algebras and Representations}, PRIN~2022S97PMY,  funded by the  Italian Ministry of University and Research.}

\begin{document}
\def\theequation{\Alph{equation}}

\begin{abstract}
We establish a bijection between torsion pairs in the category of finite-dimensional  modules over a finite-dimensional algebra $A$ and pairs $(\Z,\mathcal I)$ formed by a closed rigid set $\Z$ in the Ziegler spectrum of $A$ and a set $\mathcal I$ of indecomposable injective $A$-modules.  This can be regarded as an extension of a result from $\tau$-tilting theory \cite{AdachiIyamaReiten:14} which parametrises the functorially finite torsion pairs over $A$. We also extend a  map from \cite{DemonetIyamaJasso:19} which assigns a brick to any indecomposable $\tau$-rigid module and obtain a one-one-correspondence between finite-dimensional bricks  and certain (possibly infinite-dimensional) indecomposable modules  satisfying a rigidity condition. Our results also hold when $A$ is an artinian ring.
\end{abstract}

\maketitle
\setcounter{tocdepth}{1}
\tableofcontents

\section{Introduction}

Torsion pairs are fundamental tools in the representation theory of finite-dimensional algebras $A$.  The collection $\tors{A}$ of all torsion pairs in the category of finite-dimensional $A$-modules is a complete lattice which encodes essential information on the category.  For example, completely meet-irreducible elements of the lattice are parametrised by certain finite-dimensional modules called bricks, and bricks also label the arrows in the Hasse quiver of  $\tors{A}$ \cite{BarnardCarrollZhu:19, DemonetIyamaReadingReitenThomas:23}. Moreover, elements of $\tors{A}$  arise naturally in geometric approaches to the representation theory of $A$. Indeed, the stability conditions introduced by King \cite{King:94} to study the geometric invariant theory of quiver representations induce torsion pairs in $\mod{A}$ \cite{Bridgeland:17}.  These torsion pairs determine an equivalence relation on the set of stability conditions  that controls  an associated wall and chamber structure \cite{Asai:21, BrustleSmithTreffinger:19}.  Furthermore, via the process of HRS tilting, $\tors{A}$   can  be employed to construct abelian categories closely related and often even derived equivalent to $\mod{A}$; a famous example of this is the derived equivalence between the so-called Beilinson algebras and coherent sheaves on $\mathbb{P}^n$ \cite{Beilinson:78, Bondal:89}.

Since its introduction in the seminal paper by Adachi, Iyama and Reiten \cite{AdachiIyamaReiten:14}, $\tau$-tilting theory has proved to be an incredibly powerful tool in the study of torsion pairs and their meaning for the representation theory of $A$.  Indeed, when $\tors{A}$ is finite, the lattice is completely determined by support $\tau$-tilting modules \cite{DemonetIyamaJasso:19} and hence the features of $\mod{A}$ mentioned in the previous paragraph can be completely understood in terms of $\tau$-tilting theory: the set of bricks is finite and in bijection with indecomposable $\tau$-rigid modules \cite{DemonetIyamaJasso:19}; the wall and chamber structure is completely determined by support $\tau$-tilting modules and their mutations \cite{Asai:21}; every tilt of $\mod{A}$ is a category of modules over the endomorphism ring of a silting complex \cite{KoenigYang:14, AdachiIyamaReiten:14, DemonetIyamaJasso:19}. If, however, the set $\tors{A}$ is infinite, then there are torsion pairs that are not controlled by $\tau$-tilting theory and important representation-theoretic or geometric aspects of $A$ remain mysterious.  

In this paper, we lift the finiteness conditions imposed by $\tau$-tilting theory and study the lattice $\tors{A}$ in terms of infinite-dimensional generalisations of support $\tau$-tilting modules.  It will be more convenient, however, to work with the infinite-dimensional generalisation of the dual concept of support $\tau^-$\!-tilting modules. While the behaviour of support $\tau$-tilting and support $\tau^-$-tilting modules are completely dual, their corresponding infinite-dimensional generalisations, silting and cosilting modules \cite{AngeleriMarksVitoria:16, BreazPop:17}, have very different properties to each other. The motivation for choosing cosilting rather than silting lies in the fact that cosilting modules and their associated 2-term cosilting complexes are pure-injective objects.  We therefore have access to the techniques afforded by the theory of purity; we give an overview of the relevant parts in the appendix.

Our techniques allow us to work outside the context of finite-dimensional algebras and many of our results hold in broader generality.  Here we state our main theorems for a finite-dimensional algebra $A$ for brevity, but  most results just require $A$ to be an artinian or even a noetherian ring. 

It is already known that equivalence classes of 2-term cosilting complexes parametrise $\tors{A}$ (cf. Theorem \ref{Thm: cosilt bij}).  However, since a non-compact 2-term cosilting complex $\sigma$ in general doesn't have an indecomposable decomposition, we cannot choose a basic representative of the equivalence class.  In this paper we show that we can replace $\sigma$ by the set $\closed{\sigma}$ formed by the  indecomposable objects in $\Prod{\sigma}$.  In fact, $\closed{\sigma}$ turns out to be a closed set in a topological space associated with $A$, the Ziegler spectrum $\Zg{\Der{A}}$ of $\Der{A}$. In Theorem \ref{thm:rigidCosiltBij}, the closed sets that arise in this way are characterised as maximal rigid sets of 2-term complexes of injectives inside $\Zg{\Der{A}}$.
Using this theorem as our starting point, we lay a foundation of cosilting theory for finite-dimensional algebras and, more generally, artinian rings.

In particular, we show that every  2-term cosilting complex $\sigma$ can be uniquely represented as a  pair $(\Z,\mathcal I)$  formed by a rigid (and closed) set $\Z$ in the Ziegler spectrum $ \Zg{A}$ of the module category $\Mod{A}$ and  a subset  $\Ical$  of the set $\InjInd{A}$ of isomorphism classes of indecomposable injective $A$-modules.  The pair $(\Z,\mathcal I)$ can be obtained directly from $\closed{\sigma}$ by setting $\Z = \H{}(\closed{\sigma})$ and $\Ical = \{I\in\InjInd{A} \mid I[-1]\in\closed{\sigma}\}$.  Such pairs are characterised by the following notion.

{\bf Definition} (cf.~Definition~\ref{rpair}) A \emph{cosilting pair} 
 is a pair $(\mathcal Z,\mathcal I)$ given by a subset $\mathcal Z$ of $\Zg{A}$ and a subset  $\Ical$  of $\InjInd{A}$ which is maximal with respect to the following properties:
 \begin{itemize} 
\item the minimal injective copresentations $\mu_X$ of the modules $X$  in $\mathcal Z$ form a rigid set in $\Der{A}$;
\item $\Hom{R}(\mathcal Z,\mathcal I)=0$. 
\end{itemize}

We obtain the following generalisation of the bijections between functorially finite torsion pairs and support $\tau$-tilting pairs from \cite{AdachiIyamaReiten:14}.

{\bf Theorem A} (cf. Theorems~\ref{Thm: cosilt bij}, \ref{thm:pairModBijection} and~\ref{thm: cosilting modules and pairs})
Let $ A$ be a finite-dimensional algebra. There are bijections between
\begin{enumerate}
\item[(a)] torsion pairs  in $\mod{A}$,
\item[(b)]  torsion pairs $(\T,\F)$ in $\Mod{A}$ with $\F$ closed under direct limits,
\item[(c)] cosilting pairs $(\Z, \Ical)$ with $\Z\subseteq\Zg{A}$ and $\Ical\subseteq \InjInd{A}$.
\end{enumerate}
The bijections are given as follows.
$$\begin{array}{c|rcccc}
\text{Bijection} &&&\text{Assignment}&&  \\
\hline
 (a)\rightarrow (b) &&& t = (\t,\f)&\mapsto& (\varinjlim \t, \varinjlim \f)\\
 (b)\rightarrow (c) &&& (\T,\F) &\mapsto& (\F\cap\F^{\perp_1}\cap\Zg{A},\: \:\F^{\perp_0}\cap\InjInd{A}) \\ 
 (c)\rightarrow (a) &&& (\Z, \Ical)&\mapsto&  t= ({}^{\perp_0}\Z \cap \mod{A}, \Cogen{\Z} \cap \mod{A}) \\
 \end{array}$$ 
 
 \medskip
 
Our second main result is a large version of the brick\,--\,$\tau$-rigid correspondence from \cite{DemonetIyamaJasso:19}, where every indecomposable $\tau$-rigid module is canonically assigned a brick. In general this assignment is only an injection into the set $\brick{A}$  of finite-dimensional bricks.  By including infinite-dimensional modules, we are able to extend this to a bijection. 

Again, it is more convenient to consider the dual version. The modules that occur in the set $\Z$ for some cosilting pair $(\Z,\mathcal I)$  are the large counterpart of the indecomposable $\tau^-$-rigid modules (cf.~Proposition \ref{prop: strongly rigid} and Lemma \ref{lem: ff tau}).  Such  modules  are like grains of sediment making up the (co)silt, so we will call them  \emph{\srigid s}. In order to obtain a bijection with bricks, we have to restrict to the class $\Ri(A)$ formed by the  \srigid s $N$ such that the associated torsion pair in $\mod{A}$ with torsion class given by $\t_N=\{X\in\mod{A}\mid \Hom{A}(X,N)=0\}$ is completely meet irreducible in $\tors{A}$. The pair $(\t_N,\f_N)$ is then cogenerated  by a brick $B_N$, and this gives rise to the desired bijection. 

{\bf Theorem B} (Theorem~\ref{Thm: brick-critical correspondence} and Proposition~\ref{Prop: DIJ dual})
Let $ A$ be a finite-dimensional algebra. 
There is a bijection 
	\[\Ri(A) \to \brick{A}\]
which maps a \srigid\  $N$ to the unique brick cogenerating the torsion pair $(\t_N,\f_N)$.
The restriction of this bijection to finite-dimensional modules is dual to the  brick\,--\,$\tau$-rigid correspondence from \cite{DemonetIyamaJasso:19}.

\medskip

In a forthcoming paper \cite{AngeleriLakingSentieri:24+} we will explain how the operation of cosilting mutation introduced in \cite{ALSV} induces an operation on cosilting pairs which can be used to describe the Hasse quiver of the lattice $\tors{A}$. We will see that this operation shares important features with classical $\tau$-tilting mutation. The process of mutating at a point of a cosilting pair will consist in removing that point and replacing it in a unique way. However, unlike classical mutation, in general not every point in a cosilting pair is mutable, i.e.~can be replaced. The topology of the Ziegler spectrum will play an important role in identifying the points  at which it is possible to mutate.

\medskip

The paper is organised as follows. We collect some preliminaries in Section 2. In Section 3  we establish a one-to-one correspondence between torsion pairs in $\mod{A}$ and  maximal rigid sets of 2-term complexes of injectives inside $\Zg{\Der{A}}$. In Section 4 we introduce cosilting pairs and prove Theorem A. Section 5 is devoted to the class of \srigid s and their interaction with the brick labels in the Hasse quiver of $\tors{A}$. The proof of Theorem B is given in Section 6.  

\section{Preliminaries}
 
\subsection{Notation.} All subcategories will be strict and full. Let $ R $ be a (unital, associative) ring. By $ \Mod{R} $ we denote the category of (all left) $ R$-modules, while $ \mod{R} $ is the full subcategory of  finitely presented $ R$-modules. Moreover, we denote by $ \Der{R} $ the unbounded derived category of $ \Mod{R} $ and by $\D^\c$ the subcategory of compact objects in $ \Der{R} $ formed by the complexes which are isomorphic to bounded complexes of finitely generated projective modules.

We further denote by $\Kb{R}$ the full subcategory of $\Der{R}$ consisting of objects isomorphic to complexes of injective modules that are non-zero in only finitely many degrees.  In fact, we will mainly deal with the full subcategory $\K{R}\subseteq \Kb{R}$ of $\Der{R}$ consisting of objects isomorphic to complexes of injective $R$-modules that are non-zero in degrees $0$ and $1$ only.  If we wish to consider an $R$-module $M$ as a complex concentrated in degree $n\in \mathbb Z$, then we will write $M[n]$.  In the case where $n=0$, we will simply write $M$.
 
Given a class of objects $ \S $ in $\Der{R}$, we write  $ \Prod{\S} $ for the subcategory of $ \Der{R} $ consisting of objects (isomorphic to) direct summands of a (possibly infinite) product of modules in $ \S $.  
Moreover, given $I \subseteq \mathbb{Z}$, we denote by $\S^{\perp_I}$ the full subcategory consisting of the objects $ X $ with $ \Hom{\Der{R}}(S, X[i]) = 0 $ for all $ S \in \mathcal{S} $ and all $i\in I$. The subcategory $ {}^{\perp_I}{\S} $ is defined analogously.  We will use the notation $>0$ for the interval $\{i\in\mathbb{Z}\mid i>0\}$ and similarly for $\leq 0$.

If   $ \S $ is a class of $R$-modules, we define analogously   the subcategory $ \Prod{\S} $ of $ \Mod{R} $ and write $ \Cogen{\S} $ for the class of modules isomorphic to submodules of products of modules in $ \S $.
 Moreover, for $i \geq 0$, we denote by $\S^{\perp_i}$ the full subcategory consisting of the modules $ M $ with $ \Ext{R}^i(S, M) = 0 $ for all $ S \in \mathcal{S} $, and define analogously $ {}^{\perp_i}{\S} $.

\subsection{Torsion pairs}
\begin{definition}
A pair of subcategories $ t=\Tpair{T}{F} $ in an abelian category $ \mathcal{A} $ is  a \textbf{torsion pair} if
\begin{enumerate}
\item  $ \Hom{\mathcal{A}}(T, F) = 0 $ for all $ T \in \mathcal{T} $ and $ F \in \mathcal{F} $;
\item for all the objects $ A \in \mathcal{A} $ there exists a short exact sequence
\[
\begin{tikzcd}
0 \ar[r] & T \ar[r] & A \ar[r] & F \ar[r] & 0
\end{tikzcd}
\]
with $ T \in \mathcal{T}$ and $F \in \mathcal{F} $.
\end{enumerate}
We call $ \mathcal{T} $ the \textbf{torsion class} and $ \mathcal{F} $ the \textbf{torsionfree class} of $t$.
\end{definition}
In every torsion pair $ \Tpair{T}{F} $ the torsion and torsionfree class determine each other, in fact $ \mathcal{T}^{\perp_0} = \mathcal{F} $ and $ \mathcal{T} = {}^{\perp_0}\mathcal{F} $. Moreover, for every object $ A $ the short exact sequence in point (2) of the definition is unique up to a unique isomorphism, in particular it is possible to associate with $ A $ an object $ T_A \in \mathcal{T} $ functorially. 

For torsion pairs in module categories we have the following well-known characterisations.

\begin{proposition}\label{prop: closure properties} Let $R$ be a ring.
\begin{enumerate}
\item Let $ \mathcal{C} $ be a full subcategory of $ \Mod{R} $. Then $ \mathcal{C} $ is a torsion class if and only if it is closed under quotients, extensions and coproducts. It is a torsionfree class if and only if it is closed under submodules, extensions and products.
\item Let $\C$ be a full subcategory of $\mod{R}$.  If $R$ is left artinian, then $\C$ is a torsionfree class if and only if it is closed under subobjects and extensions.  If $R$ is left noetherian, then $\C$ is a torsion class if and only if it is closed under quotients and extensions.
\end{enumerate}
\end{proposition}

We will be interested in torsion pairs of \textbf{finite type} in $\Mod{R}$, i.e., those torsion pairs $(\T, \F)$ such that $\F$ is closed under direct limits.  Given the previous proposition, it is clear that the torsion-free classes associated to torsion pairs of finite type are exactly the subcategories of $\Mod{R}$ that are closed under submodules, extensions, products and direct limits.

{ \subsection{T-structures.}  

\begin{definition}
A pair of full subcategories $\mathbb{T} = (\X, \Y)$ of $\Der{R}$ is called a \textbf{t-structure} if \begin{enumerate} 
\item $\Hom{\Der{R}}(X, Y) = 0$ for all $X\in\X$ and $Y\in\Y$;
\item for every object $Z$ in $\Der{R}$ there is a triangle $Z_\X \to Z \to Z_\Y \to Z_\X[1]$ with $Z_\X \in \X$ and $Z_\Y \in \Y$; and
\item $X[1] \in \X$ for every $X\in \X$.
\end{enumerate}
We call $\X$ the \textbf{aisle} and $\Y$ the \textbf{coaisle} of $\mathbb{T}$. 

A t-structure $\mathbb{T} = (\X, \Y)$ is called \textbf{non-degenerate} if $\bigcap_{i\in\mathbb{Z}}\X[i] = \{0\} = \bigcap_{i\in\mathbb{Z}}\Y[i]$.
\end{definition}

It was shown in \cite{BeilinsonBernsteinDeligne:82} that, for any t-structure $\mathbb{T} = (\X, \Y)$, the full subcategory $\heart{\mathbb{T}} := \Y \cap \X[-1]$ is an abelian category called the \textbf{heart} and that there exists a cohomological functor $\H{\mathbb{T}} \colon \Der{R} \to \heart{\mathbb{T}}$.  We fix the notation $\mathrm{H}^i_\mathbb{T}$ for the composition $\H{\mathbb{T}} \circ [i]$ for all $i\in\mathbb{Z}$.

\begin{example}
The \textbf{standard t-structure} in $\Der{R}$ is denoted $\ststr{R} := (\D_R^{<0}, \D_R^{\geq 0})$ and is given by the subcategories 
	\[\D_R^{<0} = \{X \in \Der{R} \mid \mathrm{H}^i(X) = 0 \text{ for all } i\geq0\} \text{ and }\]
	\[\D_R^{\geq0} = \{X \in \Der{R} \mid \mathrm{H}^i(X) = 0 \text{ for all } i<0\}\]
where $\mathrm{H}^i$ is the standard ith cohomology functor.  The heart \[\heart{} = \{X \in \Der{R} \mid \mathrm{H}^i(X) = 0 \text{ for all } i\neq0\}\] is equivalent to $\Mod{R}$ and the cohomological functor associated to $\ststr{R}$ coincides with $\H{}$.

When the ring $R$ is clear from the context, we just write $\ststr{} = (\D^{<0}, \D^{\geq 0})$.
\end{example}

  It is well-known that t-structures with $\D^{<0} \subseteq \X \subseteq \D^{\le 0}$
 are precisely the right HRS-tilts of the standard t-structure  (see e.g.~\cite[Lem.~1.1.2]{Polishchuk:07} or \cite[Prop.~2.1]{Woolf:10}). Let us recall this construction. 

\begin{definition}\cite{HappelReitenSmalo:96}\label{HRS}
Let $\mathbb{T} = (\X, \Y)$ be a non-degenerate t-structure in $\Der{R}$ and let $t = (\T, \F)$ be a torsion pair in the heart $\heart{\mathbb{T}}$.  The \textbf{right HRS-tilt} of $\mathbb{T}$ at $t$ is the t-structure $\mathbb{T}_{t^-} = (\X_{t^-}, \Y_{t^-})$ given by 
	\[\X_{t^-} = \{ X \in \Der{R} \mid \H{\mathbb{T}}(X) \in \T \text{ and } \mathrm{H}^k_\mathbb{T}(X) = 0 \text{ for all } k>0\} \: \text{ and}\]
	\[\Y_{t^-} = \{ X \in \Der{R} \mid \H{\mathbb{T}}(X) \in \F \text{ and } \mathrm{H}^k_\mathbb{T}(X) = 0 \text{ for all } k<0\}.\]
The heart $\heart{t^-}$ of $\mathbb{T}_{t^-}$ is the full subcategory 
	\[\heart{t^-} = \{X \in \Der{R} \mid \text{ there exists a triangle } F \to X \to T[-1] \to F[1] \text{ where } F\in\F \text{ and } T\in \T\}\]
and it contains a torsion pair $\bar{t} = (\F, \T[-1])$. Let $\H{t^-}$ denote the associated cohomological functor.
\end{definition}}

\subsection{Essentials on purity.}  We recall very briefly the fundamental notions on purity we will use in the sequel. A more detailed account is given in the Appendix.

Let $R$ be a ring. A short exact sequence $0\to X\stackrel{f}{\to} Y\stackrel{g}{\to} Z\to 0$ in $\Mod{R}$ is \textbf{pure-exact} if $0\to \Hom{R}(F, X)\to \Hom{R}(F,Y)\to \Hom{R}(F,Z)\to 0$ is an exact sequence of abelian groups for every finitely presented $R$-module $F$. Similarly, a triangle $ X\stackrel{f}{\to} Y\stackrel{g}{\to} Z\to X[1]$ in $\Der{R}$ is pure-exact if $0\to \Hom{\Der{R}}(F, X)\to \Hom{\Der{R}}(F,Y)\to \Hom{\Der{R}}(F,Z)\to 0$ is an exact sequence of abelian groups  for every compact object $F\in\Der{R}$. In both cases $f$ is called a \textbf{pure-monomorphism} and $g$ a \textbf{pure-epimorphism}, and $X$ is a \textbf{pure subobject} of $Y$. 

An object $X$ in $\Mod{R}$ or $\Der{R}$ is \textbf{pure-injective} if every pure-monomorphism  starting at $X$ is a split monomorphism. 

A subcategory $\C$ of $\Mod{R}$ or $\Der{R}$ is \textbf{definable} if it is closed under direct products, pure subobjects and pure epimorphic images.

We denote by $\Zg{R}$ the set of isomorphism classes of indecomposable pure-injective objects in $\Mod{R}$.  
The subsets of $\Zg{{R}}$ of the form $\C \cap \Zg{{R}}$ where $\C$ is a definable subcategory of $\Mod{R}$ satisfy the axioms for the closed sets of a topology.  We refer to the set $\Zg{{R}}$ endowed with this topology as the \textbf{Ziegler spectrum of $\Mod{R}$}. 

The \textbf{Ziegler spectrum} $\Zg{\Der{R}}$ of $\Der{R}$  is defined correspondingly.
We will often focus on the subset   $\ZgInt{R} := \K{R}\cap\Zg{\Der{R}}$  formed by the isomorphism classes of indecomposable pure-injective complexes of injectives concentrated in degrees 0 and 1. The following result provides a link between the two spectra.

\begin{lemma} \cite[Prop.~7.1]{GarkushaPrest:05} \label{Lem: homologypi}
 If $\sigma \in \ZgInt{R}$, then the module $M:=\H{}(\sigma)$ is contained in $\Zg{R}$. If $I$ is indecomposable injective, then the stalk complexes $I[0]$ and $I[-1]$ are contained  in $\ZgInt{R}$.
\end{lemma}

\subsection{Cosilting complexes and t-structures} 

In this section  we briefly review useful results on 2-term cosilting complexes in the derived category of a ring $R$.  As we will see in the next subsection,  equivalence classes of such complexes parametrise torsion pairs of finite type in $\Mod{R}$ and, if $R$ is a left noetherian ring, these objects also parametrise the torsion pairs in the category $\mod{R}$ of finitely presented modules. For details we refer to \cite{PsaroudakisVitoria:18,MarksVitoria:18,Laking:20,ZhangWei:17,Angeleri:18}

\begin{proposition}\cite[Proposition 3.10]{MarksVitoria:18}\label{prop: cosilting definition}
The following statements are equivalent for $\sigma \in \K{R}$. \begin{enumerate}
\item The pair of full subcategories $({}^{\perp_{\leq 0}} \sigma, {}^{\perp_{> 0}} \sigma)$ is a t-structure in $\Der{R}$.
\item For all sets $I$ we have that $\Hom{\Der{R}}(\sigma^I, \sigma[1]) = 0$ and $\thick{\Prod{\sigma}} = \Kb{R}$.
\end{enumerate}
\end{proposition}

\begin{definition} We call a complex $\sigma \in \K{R}$ satisfying the equivalent conditions above a \textbf{(2-term) cosilting complex}.  We say that two such complexes $\sigma$ and $\sigma'$ are \textbf{equivalent} if they determine the same t-structure, that is, if $\Prod{\sigma} = \Prod{\sigma'}$.  We  call a t-structure of this form a \textbf{cosilting t-structure}.
\end{definition}

\begin{notation}
Given a 2-term cosilting complex $\sigma$ in $\Der{R}$, we denote the associated t-structure by $\tstr{\sigma} := ({}^{\perp_{\leq 0}} \sigma, {}^{\perp_{> 0}} \sigma)$. The heart of $\mathbb{T}_\sigma$ will be denoted by $\heart{\sigma}$ and the cohomological functor by $\H{\sigma} \colon \Der{R} \to \heart{R}$. 
\end{notation}

{We collect together some properties of  cosilting complexes that we will use in the sequel.

\begin{theorem}\cite{AngeleriMarksVitoria:17,MarksVitoria:18,Laking:20}\label{Thm: cosilt summary} 
Let $\sigma$ be a 2-term cosilting complex in $\Der{R}$.  Then the following statements hold.\begin{enumerate}
\item The complex $\sigma$ is a pure-injective object of $\Der{R}$.

\item Using the notation $\V = {}^{\perp_{> 0}} \sigma$, we have that $\Prod{\sigma} = \V \cap \V^{\perp_1}$.

\item The cohomological functor $\H{\sigma}$ induces an equivalence of categories 
	\[\H{\sigma} \colon \Prod{\sigma} \to \Inj{\heart{\sigma}}\] 
and there is a natural isomorphism
	\[\Hom{\heart{\sigma}}(\H{\sigma}(-), \H{\sigma}(\sigma)) \cong \Hom{\Der{R}}(-, \sigma).\]
	
\item There is an equivalence of categories $F \colon \Mod{\D^\c}/({}^{\perp_0}\y\sigma) \to \heart{\sigma}$ such that the following diagram commutes:
	\[\xymatrix{ \Der{R} \ar[r]^-\y \ar[d]_{\H{\sigma}} & \Mod{\D^\c} \ar[r]^-q & \Mod{\D^\c}/({}^{\perp_0}\y\sigma) \ar[dll]^-F \\ \heart{\sigma} & & }\]
where $\y$ is the restricted Yoneda functor (Definition \ref{Def: Purity Der}) and $q$ is the localisation functor.

\item The cosilting t-structures are exactly the t-structures $(\X, \Y)$ such that $$\D^{<0} \subseteq \X \subseteq \D^{\le 0}:=\D^{<0}[-1]$$ and whose heart is a Grothendieck category. In particular, every cosilting t-structure is a right HRS-tilt of the standard t-structure.
\end{enumerate}

\end{theorem}

}

\subsection{Cosilting modules and torsion pairs}

To any torsion pair  $t=(\T,\F)$ in $\Mod{R}$ we can assign the t-structure $\mathbb{\ststr}_{t^-}$ obtained as right HRS-tilt of the standard t-structure $\ststr{}$ at $t$. 
\begin{proposition}(\cite[Thm.~1.3]{ZhangWei:17}, \cite[Dual of Thm.~4.9]{AngeleriMarksVitoria:16} and \cite[Cor.~3.9]{Angeleri:18})\label{ft}
Let $R$ be a ring and let   $t=(\T,\F)$ be a torsion pair in $\Mod{R}$. The t-structure $\mathbb{\ststr}_{t^-}$ is a cosilting t-structure  if and only if $t$ is a torsion pair of finite type.
\end{proposition}

\begin{definition}
 A torsion pair as in Proposition~\ref{ft} will be called a  \textbf{cosilting torsion pair}. 
\end{definition}

Tautologically, there is a bijection between cosilting torsion pairs and equivalence classes of 2-term cosilting complexes. It maps a 2-term cosilting complex $\sigma$ to the torsion pair $ ({}^{\perp_0}C, \Cogen{C})$ cogenerated by the module $C=\H{}(\sigma)$. 
The modules $C$ that arise in this way are exactly those possessing an injective copresentation $\sigma$ that is a 2-term cosilting complex.  We recall that an injective copresentation of a module $M$ is a morphism $\sigma \colon E_0 \to E_1$ between injective modules such that $\Ker{\sigma} \cong M$.  Such a morphism can be identified with a complex in $\K{R}$ in the obvious way.

In order to characterise such modules, we need to introduce further  notation.

\begin{notation}\label{csigma}
Let $R$ be a ring and consider $\sigma \colon N_0 \to N_1$ a morphism in $\Mod{R}$.  We denote by $\C_\sigma$ the full subcategory of modules $M$ in $\Mod{R}$ such that 
	\[\Hom{R}(M, \sigma)\colon \Hom{R}(M, N_0) \to \Hom{R}(M, N_1)\] 
is an epimorphism.
\end{notation}

\begin{remark}\label{Rem: rigid C sigma}{\rm
(1) Let $\sigma, \mu \in \K{R}$.  We can consider $\sigma$ and $\mu$ as morphisms in $\Mod{R}$ between injective modules.  Since morphisms in $\Hom{\Der R}(\sigma, \mu[1])$ can be calculated as cochain maps up to homotopy, it follows that $\Hom{\Der R}(\sigma, \mu[1]) = 0$ if and only if $\H{}(\sigma)\in\C_\mu$.  See \cite[Lem.~3.3]{Angeleri:18}.

(2) Let $M$ and $N$ be $R$-modules and let $\mu_N$ be a minimal injective copresentation of $N$.  It follows from a minor adaptation of the proof of \cite[Lem.~4.13]{ZhangWei:17} that
   $ M \in \mathcal{C}_{\mu_N} $ if and only if all submodules of $M$ are contained in ${}^{\perp_1}N$.} 
\end{remark}

\begin{definition}
Let $R$ be a ring.  A module $C$ in $\Mod{R}$ is called a \textbf{cosilting module} if there exists an injective copresentation $\sigma \colon E_0 \to E_1$ of $C$ such that $\Cogen{C} = \C_\sigma$.  We say that $C$ is cosilting with respect to $\sigma$.  Cosilting modules $C$ and $D$ are declared to be \textbf{equivalent} if $\Prod{C} = \Prod{D}$.
Furthermore, an $R$-module $C$ is called a \textbf{cotilting module} if $\Cogen{C} = {}^{\perp_1}C$.
\end{definition}

Cotilting modules are always cosilting modules but the converse is not true in general.  The following proposition makes this more precise.

\begin{proposition}{\cite{ZhangWei:17}}\label{prop: cosilt is local cotilt}
The following statements are equivalent for an $R$-module $C$.
\begin{enumerate}
\item $C$ is a cosilting module over $R$.
\item $\Cogen{C}$ is a torsion-free class and $C$ is a cotilting module over $\overline{R} := R/\Ann{C}$.
\end{enumerate}
\end{proposition}

We  now summarise the discussion above.
\begin{theorem}\label{Thm: cosilt bij} Let $R$ be a ring.
There are bijections between
\begin{enumerate}
\item[(a)] cosilting torsion pairs in $\Mod{R}$,
\item[(b)] equivalence classes of 2-term cosilting complexes in $\Der{R}$,
\item[(c)] equivalence classes of cosilting modules in $\Mod{R}$,
\end{enumerate}
and when $R$ is a left noetherian ring, there is  a further bijection with
\begin{enumerate}
\item[(d)] torsion pairs  in $\mod{R}$.
\end{enumerate}
The bijections are given as follows.
$$\begin{array}{c|rcccc}
\text{Bijection} &&&\text{Assignment}&&  \\
\hline
 (a)\rightarrow (b) &&& t&\mapsto& \sigma \text{ with } \tstr{\sigma}=\mathbb{\ststr}_{t^-} \\
 (b)\rightarrow (c) &&& \sigma&\mapsto& C=\H{}(\sigma)\\
 (c)\rightarrow (a) &&& C&\mapsto&  ({}^{\perp_0}C, \Cogen{C})\\
 (d)\rightarrow (a) &&& (\t, \f) &\mapsto&  (\varinjlim(\t), \varinjlim(\f)) \\
  (a)\rightarrow (d) &&& ( \mathcal{T}, \mathcal{F} ) & \mapsto &  (\mathcal{T} \cap \mod{R} , \mathcal{F} \cap \mod{R}  )
 \end{array}$$ 

\end{theorem}
The bijection between (a) and (d) relies on a result in \cite{Crawley-Boevey:94} and  allows us to study the lattice $\tors{A}$ of an artin algebra $A$ in terms of the associated 2-term cosilting complexes and their mutations, as we will see later on in this paper and in \cite{AngeleriLakingSentieri:24+}.

\section{Closed sets associated to cosilting complexes}\label{sec: cosilting sets}

In this section, we add a further bijection to the list in Theorem~\ref{Thm: cosilt bij} by showing that every 2-term cosilting complex is completely  determined, up to equivalence, by a closed set in the Ziegler spectrum of $\Der{R}$. More precisely, we establish a one-to-one-correspondence between equivalence classes of  2-term cosilting complexes and maximal rigid (closed) subsets of $\ZgInt{R}$.

Our first aim is to show that every equivalence class of 2-term cosilting complexes over a noetherian ring uniquely determines a closed subset of the Ziegler spectrum of $\Der{R}$.  First we need a small lemma.  Recall that $\mathrm{Def}(\mu)$ denotes the smallest definable subcategory of $\Der{R}$ containing a given object $\mu\in\Der{R}$.

\begin{lemma}\label{Lem: compact}
Let $R$ be a left coherent ring.  Suppose $\mu \in \K{R}$ and $Y \in \Db{R}$.  There exists $\widetilde{Y}$ in $\D^\c$ such that $\Hom{\Der{R}}(Y, N) \cong \Hom{\Der{R}}(\widetilde{Y}, N)$ for all $N$ in  $\mathrm{Def}(\mu)$.
\end{lemma}
\begin{proof}
Let $N\in \mathrm{Def}(\mu)$.  Since $\bigcap_{i\neq 0,1}\Ker{H^i_R}$ is a definable subcategory containing $\mu$, we have that $N$ is concentrated in degrees $0$ and $1$.  We can therefore assume without loss of generality that $Y$ is a bounded above complex of finitely presented projective modules and $N$ is a bounded below complex of injective modules with zeroes in negative degrees.  Let $\widetilde{Y}$ be the brutal truncation of $Y$ at degree $-1$ and note that $\widetilde{Y}\in \D^\c$.  Since the homomorphisms $Y \to N$ and $\widetilde{Y} \to N$ may be calculated as cochain maps up to homotopy, we have that $\Hom{\Der{R}}(Y, N) \cong \Hom{\Der{R}}(\widetilde{Y}, N)$.
\end{proof}

We are now able to prove the desired result. For unexplained terminology and background on the Ziegler spectrum we refer to Appendix \ref{App: Zg}.  Notice that statement (2) below is already known: see \cite[Theorem 5.2]{Saorin:2017}. 

\begin{proposition}\label{Prop: ele cog consequences}
Let $R$ be a left noetherian ring and let $\sigma$ be a 2-term cosilting complex.  Then the following statements hold.
\begin{enumerate}
\item The set $\closed{\sigma} := \Prod{\sigma}\cap\Zg{\Der{R}}$ is a closed subset of $\Zg{\Der{R}}$.  In other words, $\sigma$ is an elementary cogenerator.
\item The heart $\heart{\sigma}$ is a locally coherent Grothendieck category.
\item The set $\closed{\sigma}$ is homeomorphic to the spectrum $ \Sp{\heart{\sigma}}$ of the heart $\heart{\sigma}$ via the cohomological functor $\H{\sigma} \colon \closed{\sigma} \to \Sp{\heart{\sigma}}$.
\item The complex $\bar{\sigma} := \prod_{\mu\in\closed{\sigma}} \mu$ is a 2-term cosilting complex equivalent to $\sigma$.
\item We have the following equality: $({}^{\perp_{\leq 0}} \sigma, {}^{\perp_{> 0}} \sigma) = ({}^{\perp_{\leq 0}} \closed{\sigma}, {}^{\perp_{> 0}} \closed{\sigma})$.
\end{enumerate}
\end{proposition}
\begin{proof}
(1) To show that $\closed{\sigma} := \Prod{\sigma} \cap \Zg{\Der{R}}$ is a closed subset of $\Zg{\Der{R}}$, it is sufficient (in fact equivalent) to show that 
	\[\mathrm{Def}(\sigma) \cap \Zg{\Der{R}} = \Prod{\sigma} \cap \Zg{\Der{R}}. \]	
	 By Theorem \ref{Thm: cosilt summary}(2) this is the same as showing that 
	\[\mathrm{Def}(\sigma) \cap \Zg{\Der{R}} = \V \cap \V^{\perp_1} \cap \Zg{\Der{R}}\]
where $\V = {}^{\perp_{> 0}} \sigma$.

The non-trivial inclusion is $\subseteq$ and so let $N\in \mathrm{Def}(\sigma) \cap \Zg{\Der{R}}$.  By \cite[Thm.~3.14]{MarksVitoria:18}, we have that $\V$ is a definable subcategory so $\mathrm{Def}(\sigma) \cap \Zg{\Der{R}} \subseteq \V \cap \Zg{\Der{R}}$ holds.  It remains to show that $N \in \V^{\perp_1}$.

Let $Y\in \V$.  Since $R$ is left noetherian, we can apply \cite[Thm.~3.2, Prop.~5.1]{MarksZvonareva:23} to obtain a directed system $\{Y_i\}_{i\in I}$ in the category of bounded complexes of finitely presented modules with $Y_i \in \V$ such that $Y$ is the homotopy colimit $\mathrm{hocolim}Y_i$ of the corresponding coherent diagram in $\Der{R}$.  Since $N[1]$ is pure-injective, we have that, if $\Hom{\Der{R}}(Y_i, N[1]) = 0$ holds for all $i\in I$, then $\Hom{\Der{R}}(Y, N[1]) = 0$ holds (see, for example, the proof of \cite[Prop.~4.5]{Laking:20}).

For each $Y_i$, consider the object $\widetilde{Y_i[-1]} \in \D^\c$ given by Lemma \ref{Lem: compact}.  Clearly $(\widetilde{Y_i[-1]})^{\perp_0}$ is a definable subcategory and $\sigma \in (\widetilde{Y_i[-1]})^{\perp_0}$ by the lemma and {definition of $\V$}.  Therefore 
	\[N \in \mathrm{Def}(\sigma) \subseteq (\widetilde{Y_i[-1]})^{\perp_0}\]
and so, by another application of Lemma \ref{Lem: compact}, we have $\Hom{\Der{R}}(Y_i, N[1]) = 0$ as desired.

(2)  The restricted Yoneda functor $\y:\Der{R}\to \Mod{\D^\c}$ (cf.~Definition~\ref{Def: Purity Der}) maps the pure-injective object $\sigma$ to an injective object $E:=\y\sigma$ in the functor category $\Mod{\D^\c}$ and induces a homeomorphism
$ \y \colon \Zg{\Der{R}} \to \Sp{\Mod{\D^\c}}$ by Theorem~\ref{Thm: Zg Der}. 
In particular,  $\y$ takes the closed set $\closed{\sigma}$ from (1)  to 
the closed set $\Z_\sigma:=\Prod E\cap\Sp{\Mod{\D^\c}}$. 
This shows that $E$ is an elementary cogenerator, or in other words, that $(\T,\F):=({}^{\perp_0}E, \Cogen{E})$ is a hereditary torsion pair of finite type in $\Mod{\D^\c}$, cf.~Proposition \ref{Prop: ele cogen}. From Theorem \ref{Thm: serre bij}  we infer that  
 $\Mod{\D^\c}/\T$ is a locally coherent Grothendieck category, and the statement follows from
the equivalence  $\Mod{\D^\c}/\T \simeq \heart{\sigma}$ in Theorem \ref{Thm: cosilt summary}(4).	

(3) By Theorem~\ref{Thm: fundamental Zg}(2) the spectrum of $\Mod{\D^\c}/\T$ is homeomorphic to $\closed{\sigma}$ via the composite of $\y$ with the quotient functor $q$. We now obtain the statement by applying the equivalence   $F \colon \Mod{\D^\c}/\T \to \heart{\sigma}$ and using the commutative diagram in Theorem \ref{Thm: cosilt summary}(4).

(4) and (5): Applying Remark~\ref{prods} to the elementary cogenerator $E=\y\sigma$ and the closed subset $\Z_\sigma$ of $\Sp{\Mod{\D^\c}}$, we obtain $\Prod{\y\sigma}=\Prod{\Z_\sigma}$. This shows that $\Prod\sigma=\Prod{\closed{\sigma}}$, and the statements follow immediately.
\end{proof}

\begin{corollary}
\label{lem:cosiltElemCog}
Let $ R $ be a left noetherian ring and let $ C \in \Mod{R} $ be a cosilting module. Then $\Z_C:= \Prod{C} \cap \Zg{R}$ is a closed subset of $\Zg{R}$. In other words, $C$ is an elementary cogenerator.
\end{corollary}
\begin{proof}
Since for any ideal $ I $ the subcategory $ \Mod{R/I} $ is definable in $ \Mod{R} $ and every cosilting module is cotilting in a subcategory of this form by Proposition \ref{prop: cosilt is local cotilt}, it is enough to show the statement for  cotilting modules.

Notice that a cotilting module $C$, considered as an object of $\Der{R}$ concentrated in degree zero, is isomorphic to a 2-term cosilting  complex. By Proposition~\ref{Prop: ele cog consequences}(1), we have that 
$\Prod{C}\cap\Zg{\Der{R}}$ is a closed subset of $\Zg{\Der{R}}$. We can then argue as in  \cite[proof of Thm.~7.3(2)]{GarkushaPrest:05} to conclude  that $ \Prod{C} \cap \Zg{\Der{R}} \cap \Zg{R}$  is closed in $\Zg{R}$.
 But, this is the same as  $ \Prod{C} \cap \Zg{R} $, and so our claim is proven.
Finally, by Theorem \ref{Thm: Zg Mod}, our claim means that $C$ is an elementary cogenerator.
\end{proof}

With Proposition~\ref{Prop: ele cog consequences}  we showed that the equivalence class of a given 2-term cosilting complex $\sigma$ in the derived category $\Der{R}$ of a noetherian ring $R$ is completely determined by the set $\closed{\sigma} := \Prod{\sigma}\cap \Zg{\Der{R}}$ of indecomposable pure-injective objects of $\Der{R}$.  
We now determine the closed subsets of $\Zg{\Der{R}}$ that arise in this way when the ring is left artinian.  We begin by noting that  $\K{R}$ is closed under products and summands (for any ring $R$, cf.~\cite[Lem.~2.6]{IyamaYang:18}), and so $\closed{\sigma}$ is contained in $ \ZgInt{R}$ for any 2-term cosilting complex $\sigma$.  

\begin{definition}\label{cosset} 
We say that an object $M$ in $\Der{R}$ is \textbf{rigid} if $\Hom{\Der{R}}(M, M[1]) = 0$. 
Furthermore, a set $\M$ of indecomposable objects in $\Der{R}$ is \textbf{rigid} if $\Hom{\Der{R}}(M, N[1]) = 0$ for all $M, N \in \M$.
If $\N \subseteq \ZgInt{R}$ is a rigid subset  that is maximal amongst all rigid  subsets  of $\ZgInt{R}$, we call $\N$ a \textbf{maximal rigid set}.\end{definition}

We want to show that $\closed{\sigma}$ is a maximal rigid set for any 2-term cosilting complex $\sigma$ over a left artinian ring.  In order to do this we  need the following result concerning the   subcategory $\C_\sigma$ of $\Mod{R}$ introduced in Notation~\ref{csigma}.

\begin{lemma}\label{cor:cSigmaCosilting}
Let $ \sigma : E_0 \to E_1 $ be a 2-term complex of injective modules in $ \Der{R} $, and assume that $ H^0(\sigma) $ pure-injective (in particular, this holds when $\sigma$ is  pure-injective). Then $ \mathcal{C}_\sigma \subseteq \Mod{R} $ is closed under filtered colimits. Moreover, if $ R $ is left artinian, then $ \mathcal{C}_\sigma $ is a cosilting torsion-free class in $ \Mod{R} $. 
\end{lemma}
\begin{proof}
Denote by $M= \H{}(\sigma) $  the zeroth cohomology of $ \sigma $, and assume that $M$ is pure-injective. This is always true when
 $\sigma$ is a pure-injective complex by Lemma~\ref{Lem: homologypi}.

The same argument as in the proof of \cite[Lemma 4.2]{AngeleriSentieri:23+} shows that $\C_\sigma$ is closed under direct limits. 

Now, if we assume that $ R $ left artinian, we have that a subcategory of  $ \mod{R} $ is a torsionfree class if and only if it is closed under submodules and extensions, cf.~Proposition~\ref{prop: closure properties}. Since  $ \mathcal{C}_\sigma  \cap \mod{R} $ has these properties and 
  $ \mathcal{C}_\sigma $ coincides with its direct limit closure, we infer from the bijection between (a) and (d) in Theorem~\ref{Thm: cosilt bij}  that 
 $ \mathcal{C}_\sigma $ is a cosilting torsionfree class. \end{proof}

We have just seen that the subcategory $\C_\sigma$ given by an object $\sigma$ in $\ZgInt{A}$ over a left artinian ring $A$ is a torsionfree class. In the next two results, we collect some consequences.

\begin{lemma}\label{lem:orthInProdRigSys}
Let $ R $ be a left artinian ring and let $ \mathcal{N} \subseteq \ZgInt{R}$ be a rigid set in $ \Der{R} $. Then for all $ \sigma_1, \sigma_2 \in \Prod{\mathcal{ N }} $ we have $ \Hom{ \Der{R} }(\sigma_1, \sigma_2[1]) = 0 $.
\end{lemma}
\begin{proof}

It is enough to show that, for $ \sigma \in \Prod{\mathcal{N}} $ and $ \omega \in \mathcal{N} $, we have $ \Hom{\Der{R}}(\sigma, \omega[1] ) = 0 $.  By Remark \ref{Rem: rigid C sigma} this amounts to showing that
$ \H{}(\sigma) \in \mathcal{C}_{\omega} $. 
Since $ \omega $ is pure-injective, we know from  Lemma \ref{cor:cSigmaCosilting} that $ \mathcal{C}_{\omega} $ is closed under products. 
Now assume $ \sigma $ is a direct summand of $ \prod_j \omega_j $ with $ \omega_j \in \mathcal{N} $. By hypothesis $ \Hom{\Der{R}}(\omega_j, \omega[1]) = 0 $. Thus $ \H{}(\prod_j\omega_j) = \prod_j \H{}(\omega_j) \in \mathcal{C}_\omega $, and consequently $ \H{}(\sigma) \in \mathcal{C}_{\omega} $.
\end{proof}

\begin{lemma}\label{lem: completion}
Let $\sigma$ be an object in $\K{R}$.  If $\sigma$ is rigid and $\C_\sigma$ is a torsion-free class, then there exists an object $\rho$ in $\K{R}$ such that $\sigma\oplus\rho$ is a 2-term cosilting complex and $\C_\sigma = \C_{\sigma\oplus\rho}$.  
\\
In particular, if the ring $R$ is left artinian, then $\sigma$ is rigid and $ H^0(\sigma) $ is pure-injective if and only if $\sigma$ is a direct summand of a 2-term cosilting complex.
\end{lemma}
\begin{proof}
The first part of the proposition states the existence of a Bongartz completion and is a rephrasing of {\cite[Prop.~3.10]{ZhangWei:17}}.  The second part follows from Lemma \ref{cor:cSigmaCosilting}.
\end{proof}

We are now ready to prove the fundamental correspondence between maximal rigid sets and 2-term cosilting complexes. In essence this result says that we can study the indecomposable objects in $\Prod{\sigma}$ in place of  cosilting complexes $\sigma$.

\begin{theorem}
\label{thm:rigidCosiltBij}
Let $ A $ be a left artinian ring. The assignments
\begin{align*}
\sigma & \mapsto \mathcal{N}_{\sigma} \\
\mathcal{N} & \mapsto \sigma_{\mathcal{N}} := \prod_{\omega \in \mathcal{N}} \omega
\end{align*}
are mutually inverse bijections between:
\begin{itemize}
\item equivalence classes of 2-term cosilting complexes in $ \Der{A} $; and
\item maximal rigid sets  in $\ZgInt{A}$.
\end{itemize}
\end{theorem}
\begin{proof}
We show that  $\closed{\sigma}$ is a maximal rigid set for every 2-term cosilting complex $\sigma$. 
Recall that $\Hom{\Der{A}}(\sigma^I, \sigma[1]) = 0$ for any set $I$. It follows that $ \Hom{ \Der{A} }(\nu_1, \nu_2[1]) = 0 $ for all $ \nu_1, \nu_2 \in \mathcal{N}_\sigma $, showing that $ \mathcal{N}_\sigma$ is rigid.
It remains to show maximality. Let $ \mathcal{L} \subseteq \ZgInt{A}$ be a rigid  set of $\Zg{\Der{A}}$ with $\closed{\sigma} \subseteq \mathcal{L} $. 
Pick $ \lambda \in \mathcal{L}  $.  We have that $ \Hom{\Der{A}}(\lambda, \mu[1]) = 0 $ for all $\mu \in \closed{\sigma}$, and since these are 2-term complexes of injectives, $ \lambda \in {}^{\perp_{>0}} \closed{\sigma} = {}^{\perp_{>0}} \sigma $, where the last equality holds  by Proposition \ref{Prop: ele cog consequences}(5). 
Now we can apply \cite[Lem.~2.12]{ZhangWei:17} to find two triangles in $ \Der{A} $
\[
\begin{tikzcd}[row sep=small]
\lambda = \lambda_0 \arrow[r] & \xi_0 \arrow[r] & \lambda_1 \arrow[r] & \lambda_0[1] \\
\lambda_1 \arrow[r] & \xi_1 \arrow[r] & \lambda_2 \arrow[r] & \lambda_1[1]
\end{tikzcd}
\]
with $ \xi_0, \xi_1 \in \Prod{\sigma} = \Prod{\mathcal{N}_\sigma} $ and $ \lambda_2 \in \mathrm{D}^{\ge 0} $.
Applying $ \Hom{ \Der{A} }(-, \lambda) $ to the second triangle we obtain the exact sequence
\[
\Hom{ \Der{A} }(\xi_1, \lambda[1]) \to \Hom{ \Der{A} }(\lambda_1, \lambda[1]) \to \Hom{ \Der{A} }(\lambda_2, \lambda[2]).
\]
Note that $ \Hom{ \Der{A} }(\xi_1, \lambda[1]) = 0 $ by Lemma \ref{lem:orthInProdRigSys}, because $ \xi_1 \in \Prod{\mathcal{N}_\sigma} \subseteq \Prod{\mathcal{L}} $. Moreover, $ \Hom{ \Der{A} }(\lambda_2, \lambda[2]) = \Hom{ \mathrm{K}^b(\mathrm{Inj}(A)) }( \lambda_2, \lambda[2] ) = 0 $ since $ \lambda \in \K{A}$.

It follows that $ \Hom{ \Der{A} }(\lambda_1, \lambda[1]) = 0 $, so that the first triangle must split and $ \lambda $ is a summand of $ \xi_0 \in \Prod{\mathcal{N}_\sigma} $. By Proposition \ref{Prop: ele cog consequences}, the set $\closed{\sigma} \subseteq \ZgInt{A}$ is a closed subset of $\Zg{\Der{A}}$, thus it is of the form $\closed{\sigma}=\C\cap \Zg{\Der{A}}$ for a definable subcategory $\C$ of $\Der{A}$. Since $\C$ is closed under direct products and direct summands, we infer that also  $\lambda$ belongs to  $\closed{\sigma}$.
This shows that
$\mathcal{L} = \mathcal{N}_\sigma $, and the maximality of $\closed{\sigma}$ is verified.

Next, we verify that $\sigma=\sigma_\N$ is a cosilting complex for every cosilting closed set $\N$.  By Lemma \ref{lem:orthInProdRigSys}, we have that $\Hom{\Der{A}}(\sigma^I, \sigma[1]) = 0$ for any set $I$.  In particular, the complex $\sigma_\N$ is rigid and so, by Lemma \ref{lem: completion}, there exists $\rho \in \K{A}$ such that $\gamma = \sigma_\N\oplus\rho$ is a 2-term cosilting complex.  Moreover $\N \subseteq \closed{\gamma}$ and so $\N = \closed{\gamma}$ by maximality of $\N$.  By Proposition \ref{Prop: ele cog consequences}(4), we have that $\Prod{\closed{\gamma}} = \Prod{\gamma}$ and so $\Prod{\sigma_\N} = \Prod{\N} = \Prod{\closed{\gamma}} = \Prod{\gamma}$.  It follows that $\sigma_\N$ satisfies the conditions in Proposition \ref{prop: cosilting definition}(2) and so it is a 2-term cosilting complex.  

So the assignments $\sigma \mapsto \closed{\sigma}$ and $\N\mapsto\sigma_\N$ are well-defined, and they are mutually inverse bijections as a consequence of Proposition \ref{Prop: ele cog consequences}(4).
\end{proof}

Combining the theorem with Proposition \ref{Prop: ele cog consequences}, we obtain the following corollary.

\begin{corollary}\label{cor: cosilting sets are closed}
Let $A$ be a left artinian ring.  Every maximal rigid set $\N$ in $\Der{A}$ is a closed subset of the Ziegler spectrum $\Zg{\Der{A}}$ of $\Der{A}$ and $({}^{\perp_{\leq 0}}\N, {}^{\perp_{> 0}}\N)$ is a t-structure in $\Der{A}$.
\end{corollary}

Another consequence of the last theorem is that we can replace the equivalence classes of cosilting complexes in item (b) of Theorem~\ref{Thm: cosilt bij} by maximal rigid sets. In fact, we note that each of the sets below admits a partial order turning it into a complete lattice, and the bijections in Theorems~\ref{Thm: cosilt bij} and \ref{thm:rigidCosiltBij} respect the ordering.

\begin{itemize}
\item Let $\tors{A}$ denote the set of torsion pairs in $\mod{A}$ with the following ordering:
	\[(\u,\v) \leq (\t,\f) \hspace{5mm} \Leftrightarrow \hspace{5mm}  \u \subseteq \t.\]
\item Let $\Cosilt{A}$ denote the set of cosilting torsion pairs in $\Mod{A}$ with the following ordering: 
	\[(\U,\V) \leq (\T,\F) \hspace{5mm} \Leftrightarrow \hspace{5mm}  \U \subseteq \T.\]
\item Let $\CosiltZg{A}$ denote the set of maximal rigid  sets in $\ZgInt{A}$ with the following ordering:
	\[ \M \leq \N \hspace{5mm} \Leftrightarrow \hspace{5mm}  \Hom{\Der{A}}(\N, \M[1]) = 0. \]
\end{itemize}

\begin{remark}\label{rem: order aisles}{\rm
The ordering on $\CosiltZg{A}$ corresponds to the ordering of the associated t-structures by inclusion of the aisles.  Indeed, suppose $\M \leq \N$.  Then ${\Hom{\Der{A}}(\N, \M[1]) = 0}$ implies that $\N \subseteq {}^{\perp_{>0}}\M$.  By \cite[Prop.~4.9]{PsaroudakisVitoria:18} we have that ${}^{\perp_{>0}}\N \subseteq {}^{\perp_{>0}}\M$ and therefore ${}^{\perp_{\leq0}}\M \subseteq {}^{\perp_{\leq0}}\N$.  }
\end{remark}

\begin{corollary}
\label{cor: CB and cosilting set bijections}
Let $A$ be a left artinian ring. There exist order-preserving bijections
	\[\tors{A} \overset{\phi}{\rightarrow} \Cosilt{A} \overset{\psi}{\rightarrow} \CosiltZg{A}\]
where $\phi$ is the bijection given in Theorem~\ref{Thm: cosilt bij}, and $\psi$ is a composition of the bijections in Theorems~\ref{Thm: cosilt bij} and \ref{thm:rigidCosiltBij}, that is, $\psi$  maps a cosilting torsion pair $t=(\T,\F)$ to the maximal rigid set $\N_{\sigma}$ which is associated to the cosilting complex $\sigma$ corresponding to $t$.
Moreover, the inverse assignment of $\psi$ is given by $\psi^{-1}(\N) = ({}^{\perp_0}\H{}(\N), \Cogen{\H{}(\N)})$.
\end{corollary}
\begin{proof}  First we note that the composition of the bijections in Theorems~\ref{Thm: cosilt bij} and \ref{thm:rigidCosiltBij} yields that $\psi^{-1}(\N) = ({}^{\perp_0}C_\N, \Cogen{C_\N})$ where $C_\N :=\prod_{\mu\in\N}\H{}(\mu)$.  Since $\Prod{\H{}(\N)} = \Prod{C_\N}$, this yields that $\psi^{-1}(\N) = ({}^{\perp_0}\H{}(\N), \Cogen{\H{}(\N)})$.  We already know that $\phi$ and $\psi$ are well-defined bijections and it remains to prove that they are order-preserving.

The  bijection $\phi$ is clearly order-preserving. 
Now let  $u=(\U,\V) \leq t=(\T,\F)$ be two torsion pairs in $\Cosilt{A}$ with images $\M=\psi(u)$ and $\N=\psi(t)$. According to the bijections in Theorem~\ref{Thm: cosilt bij}, the t-structures $\tstr{\sigma_\M}$ and $\tstr{\sigma_\N}$ are precisely the right HRS-tilts of the standard t-structure at $u$ and $t$, respectively. By construction, when 
$\U\subseteq\T$, the aisle of the right HRS-tilt at $u$ is contained in the aisle of the right HRS-tilt at $t.$  By Remark~\ref{rem: order aisles} we infer that $\M\le \N$. 
\end{proof}

\section{Cosilting pairs in the module category}

Throughout this section we will use $A$ to denote a left artinian ring.  
In Corollary \ref{cor: CB and cosilting set bijections} we saw that the cosilting torsion pairs in $\Mod{A}$ are exactly those of the form $t = ({}^{\perp_0}\H{}(\N), \Cogen{{\H{}(\N)}})$ for a maximal rigid set $\N$ in $\Der{A}$ and that $\N$ can be recovered from $t$.  In other words, the set $\N$ is determined by the indecomposable pure-injective modules in $\H{}(\N)$.  In this section we characterise the sets of modules that arise in this way. The resulting notion will in fact consist of a pair of sets of indecomposable modules; this approach is inspired by the notion of a support $\tau$-tilting pair (or rather its dual version) introduced in \cite{AdachiIyamaReiten:14}.  

\begin{notation}\label{Not: minimal inj cop}
From now on we will use the notation $\mu_M$ for the 2-term complex given by the minimal injective copresentation of a module $M$. Moreover, we denote by  $\InjInd{A}$ the set of isomorphism classes of indecomposable injective $A$-modules. 
\end{notation}

\begin{definition}\label{rpair} A  \textbf{rigid pair} is a pair $(\mathcal Z,\mathcal I)$ such that  
\begin{enumerate}
\item[(i)]  $\mathcal Z$ is a subset of $\Zg{A}$, and   $\Ical$ is a subset of $\InjInd{A}$; 
\item[(ii)] the complexes $\mu_X$ such that $X$ is in $\mathcal Z$, form a rigid set in $\Der{A}$;
\item[(iii)] $\Hom{A}(\mathcal Z,\mathcal I)=0$. 
\end{enumerate}
A rigid pair $ \Tpair{Z}{I} $ is a \textbf{cosilting pair} if it satisfies the following maximality condition
\begin{itemize}
\item[(iv)]  every rigid pair $ \Tpair{Z'}{I'} $ with $ \Z \subseteq \Z' $ and $ \mathcal{I} \subseteq \mathcal{I}' $ equals  $ \Tpair{Z}{I}$.
\end{itemize}
We denote the set of cosilting pairs over $A$ by $\Cosiltpair{A}$.
\end{definition}

Our next aim is to prove the following result.
 \begin{theorem}
\label{thm:pairModBijection}
Let $ A$ be a left artinian ring. There exists a bijection
\[\Cosiltpair{A}\rightarrow \tors{A}\] 
which takes a cosilting pair  $(\Z, \Ical)$
to the torsion  pair $t= ({}^{\perp_0}\Z \cap \mod{A}, \Cogen{\Z} \cap \mod{A})$.
\end{theorem}

We will prove the theorem by establishing in Proposition~\ref{prop: bij rigid} a bijection between cosilting pairs and maximal rigid sets in $\ZgInt{A}$ and composing it with the bijection in Corollary \ref{cor: CB and cosilting set bijections}.

First we observe that the complexes contained in $\ZgInt{A} = \K{A}\cap\Zg{\Der{A}}$ are determined by subsets of $\Zg{A}$ and  $\InjInd{A}$.  As usual we will identify the elements of these sets with an arbitrary representative of the corresponding isomorphism class.

\begin{lemma}\label{Lem: complexes determined by modules}
The following statements hold.
\begin{enumerate}
\item If $\sigma$ is a 2-term complex of injective $A$-modules and $M:= \H{}(\sigma)$, then $\sigma$ is isomorphic (as a complex) to $\mu_M \oplus \nu \oplus I[-1]$, where $\nu$ is an isomorphism and $ I $ an injective module.
\item $\ZgInt{A} \subseteq \{\mu_M \mid M\in\Zg{A}\} \sqcup \{I[-1]\mid I\in \InjInd{A}\}$.
\end{enumerate}
\end{lemma}
\begin{proof} (1) follows easily from the defining properties of $\mu_M$.

(2) Let $\sigma\in\ZgInt{A}$.  Since $\sigma$ is indecomposable, it follows from (1) that $\sigma$ is either isomorphic to $I[-1]$ for some $I\in \InjInd{A}$ or to $\mu_M$ for some indecomposable $A$-module $M$.  In the latter case the module $M = \H{}(\mu_M)$ is contained in $\Zg{A}$ by Lemma~\ref{Lem: homologypi}.
\end{proof}

\begin{lemma}\label{lem: 2-term Zg to modules Zg}
Let $R$ be a ring. If $M\in\Mod{R}$ has local endomorphism ring, so has the 2-term complex $\mu_M$. In particular, if $M \in \Zg{R}$, then  $ \mu_M $ is indecomposable both as a complex and as an object of $\Der{R}$.
\end{lemma}
\begin{proof}
Assume that $\mathrm{End}_{R}(M)$ is local, and let $ S = \mathrm{End}_{\mathrm{Ch}(R)}(\mu_M) $ be the endomorphism ring of $\mu_M$ in the category of chain complexes $\mathrm{Ch}(R)$. We have to show that for every $ \phi \in S $ either $ \phi $ or $ 1_S - \phi $ is invertible. Notice that any such $\phi$ restricts to an endomorphism $f\in\mathrm{End}_{R}(M)$, that is, we have a commutative diagram
\[
\begin{tikzcd}
0 \ar[r] & M \ar[r, "i_0"] \ar[dd, dashed, "f"] & I_0 \ar[dd, "f_0"] \ar[rr, "\mu_M"] \ar[rd,  two heads, "p_0"] & & I_1  \ar[dd, "f_1"] \\
& & & Z_0 \ar[ru,  hook, "i_1"] \ar[dd, dashed,  bend left, "z"]
\\
0 \ar[r] & M \ar[r, "i_0"] & I_0 \ar[rd,  two heads, "p_0"'] \ar[rr, "\mu_M"] & &  I_1 \\
& & & Z_0 \ar[ru,  hook, "i_1"']
\end{tikzcd}
\]
where $ f_0 $ and $ f_1 $ are the components of $\phi$, the map $ f $ is the unique factorization of $ f_0 \circ i_0 $ through the kernel $ i_0 $ of $ \mu_M $, and $ z : Z_0 \to Z_0 $ is the cokernel factorization of $ p_0 \circ \f_0 $. 

We claim that if $ f $ is an isomorphism, then $ \phi $ is invertible. 
To this end, assume that $g$ is the inverse map of $f$. By the injectivity of $I_0$, there is a map $ g_0 : I_0 \to I_0 $ such that $ g_0 \circ i_0 = i_0 \circ g $. 
Thus, $ g_0 \circ f_0 \circ i_0 = g_0 \circ i_0 \circ f = i_0 \circ g \circ f = i_0 \circ 1_M = i_0 $. Now left minimality of $ i_0 $ yields that $ g_0 \circ f_0 $ is an isomorphism. In the same way we verify that $ f_0 \circ g_0 $ is an isomorphism, thus the induced map $ f_0 $ is an isomorphism.
In fact, we can use the inverse map $ f_0^{-1} $ as a factorisation in place of a generic $ g_0 $, indeed:
 $ f_0^{-1}i_0f = f_0^{-1}f_0i_0 = i_0 = i_0gf $, which yields $f_0^{-1}i_0 = i_0g $.
Moreover, $z$ is an isomorphism, and
the cokernel factorisation of $ p_0 \circ g_0 $ yields a map $ \gamma : Z_0 \to Z_0 $ such that $ \gamma \circ p_0 = p_0 \circ g_0 $. Then $ \gamma \circ z \circ p_0 = \gamma \circ p_0 \circ f_0 = p_0 \circ g_0 \circ f_0 = p_0 $, thus $\gamma = z^{-1}$. 

Now, repeating the same argument, starting with $ z $ and using the left minimality of $ i_1 $, we see that that $ f_1 $ is also an isomorphism and that the induced map $ g_1 $ can be chosen to be $ f_1^{-1} $. This shows that the map $ \phi $ is invertible and yields  the claim.

At this point, let us assume that $ \phi \in S $ is not invertible. It follows that the induced endomorphism $f$ can't be invertible, and by assumption on  $\mathrm{End}_{R}(M)$ we infer that $1_M - f$ is  invertible.  Notice that the map $ 1_{S} - \phi $ restricts to $1_M - f$.  By the previous claim we conclude that 
 $ 1_S - \phi $ is invertible. 

We have shown that the endomorphism ring $S$ is local, thus the complex $ \mu_M $ is indecomposable. Moreover, as the endomorphism ring of the object $\mu_M$ in the derived category is a quotient of the ring $ S $, we infer that this new ring is again local, so that $ \mu_{M} $ is also indecomposable in $\Der{R}$.
For the last statement, we just need to recall that the endomorphism ring of any indecomposable pure-injective module is local.
\end{proof}

\begin{notation}\label{not: bij cosilting sets pairs} As a consequence of Lemma \ref{Lem: complexes determined by modules}, every subset $\N \subseteq\ZgInt{A}$ uniquely determines the following pair consisting of a subset of $\Zg{A}$ and a subset of $\InjInd{A}$: \[(\Z_\N, \Ical_\N) := (\{M\in \Zg{A} \mid \mu_M\in\N\}, \{I\in \InjInd{A}\mid I[-1]\in\N\}).\] 
We note that $\Z_\N$ coincides with $\H{}(\N)$.
Conversely, for any pair $(\Z, \Ical)$ with $\Z\subseteq \Zg{A}$ and $\Ical\subseteq\InjInd{A}$, we use the following notation for the corresponding subset of $\ZgInt{A}$: \[\N_{(\Z, \Ical)} := \{ \mu_M \mid M\in \Z\} \sqcup \{I[-1]\mid I\in \Ical\}.\]
\end{notation}

\begin{proposition}\label{prop: bij rigid}
Let $A$ be a left artinian ring. Then the assignments \[\N \mapsto (\Z_\N, \Ical_\N) \text{  and  } (\Z, \Ical) \mapsto \N_{(\Z, \Ical)}\] define mutually inverse bijections between the set of rigid subsets of $\ZgInt{A}$ and the set of rigid pairs $(\Z, \Ical)$ with $\Z\subseteq\Zg{A}$ and $\Ical\subseteq \InjInd{A}$.
Moreover, the assignments restrict to mutually inverse bijections between the set of maximal rigid sets in $\ZgInt{A}$ and the set of all cosilting pairs over $A$.
\end{proposition}
\begin{proof}
Clearly the assignments are mutually inverse, so it suffices to show that they are well-defined.  

Suppose $\N$ is a rigid set in $\ZgInt{A}$.  We must show that the conditions (i)-(iii) of Definition \ref{rpair} hold for $(\Z_\N, \Ical_\N)$.  The first two conditions are clear, we verify  (iii).  A straightforward argument shows that any morphism $Z \to I$ with $Z\in\Z_\N$ and $I\in \Ical_\N$ yields a morphism $\mu_Z \to I$ in $\Der{A}$, and $\Hom{A}(Z, I) =0$ if and only if every morphism of compexes $\mu_Z \to I$ is null-homotopic. The latter condition holds due to the fact that $\N$ is rigid: we have that $\Hom{\Der{A}}(\mu_Z, I) \cong \Hom{\Der{A}}(\mu_Z, (I[-1])[1])  = 0$.

Now we consider a rigid pair $(\Z, \Ical)$ and show that the set $\N_{(\Z, \Ical)}$ is rigid. We first  prove that every object in $\N_{(\Z, \Ical)}$ is indecomposable and pure-injective.  The complexes $I[-1]$ are clearly all indecomposable and they are also pure-injective by Lemma~\ref{Lem: homologypi}.  The complexes $\mu_M$ are all indecomposable summands of cosilting complexes due to Lemma \ref{lem: completion} and Lemma \ref{lem: 2-term Zg to modules Zg}, and so they are indecomposable pure-injective  by Theorem \ref{Thm: cosilt summary}.
Next, we verify rigidity. Pick  $M,N \in \Z$ and $I, J\in\Ical$.  By definition we have that $\Hom{\Der{A}}(\mu_N, \mu_M[1])=0=\Hom{\Der{A}}(\mu_N, (I[-1])[1])$.  Moreover $\Hom{\Der{A}}(I[-1], (J[-1])[1]) \cong \Ext{A}(I, J) =0$ because $J$ is an injective module.  Finally $\Hom{\Der{A}}(I[-1], \mu_M[1]) =0$ for degree reasons.  
  
Since maximal rigid sets and cosilting pairs are defined to be maximal among rigid sets and rigid pairs respectively, it is clear that the bijection between rigid sets and rigid pairs restricts to a bijection between maximal rigid sets and cosilting pairs.
\end{proof}

We saw in Section \ref{sec: cosilting sets} that there is a partial order on $\CosiltZg{A}$. The bijection in Proposition \ref{prop: bij rigid} now induces a partial order on $\Cosiltpair{A}$. 

\begin{corollary}\label{cor: cosilt pair poset}
The set $ \Cosiltpair{A} $ of cosilting pairs over $A$ admits the following partial order 
	\[   (\Z', \mathcal{I}')\leq (\Z, \mathcal{I}) \hspace{5mm} \Leftrightarrow \hspace{5mm}  \Cogen{\Z} \subseteq \Cogen{\Z'} . \]
\end{corollary}
\begin{proof}
 Consider $(\Z, \mathcal{I})$ and $(\Z', \mathcal{I}')$ in $\Cosiltpair{A}$ and the corresponding maximal rigid sets $\N:=\N_{(\Z, \mathcal I)}$ and $\N'=\N_{(\Z', \mathcal I')}$.  By Corollary~\ref{cor: CB and cosilting set bijections}, we have that $\N' \leq \N$ if and only if $\Cogen{\H{}(\N)}\subseteq\Cogen{\H{}(\N')}$.  Since $\H{}(\N) = \Z$ and $\H{}(\N') = \Z'$, we have that the induced partial order on $\Cosiltpair{A}$ is the one stated.
\end{proof}

\begin{proof}[Proof of Theorem \ref{thm:pairModBijection}]
Combining Proposition~\ref{prop: bij rigid}  with Corollary~\ref{cor: CB and cosilting set bijections} we obtain a bijection
	\[\Cosiltpair{A}\longrightarrow \CosiltZg{A} \overset{\psi^{-1}}{\longrightarrow}\Cosilt{A}\overset{\phi^{-1}}{\longrightarrow} \tors{A} \]
which takes a cosilting pair $(\Z, \Ical)$ first  to  $\N:=\N_{(\Z, \Ical)} = \{ \mu_M \mid M\in \Z\} \sqcup \{I[-1]\mid I\in \Ical\}$  and then to the torsion pair  $\psi^{-1}(\N)=({}^{\perp_0}\H{}(\N), \Cogen{\H{}(\N)})= ({}^{\perp_0}\Z , \Cogen{\Z} )$, and finally $\phi^{-1}$  restricts this torsion pair to $\mod{A}$. So the composite assignment is 
\[\Cosiltpair{A}\longrightarrow  \tors{A}, \; (\Z,\mathcal I)\mapsto t=({}^{\perp_0}\Z \cap \mod{A}, \Cogen{\Z} \cap \mod{A}).\] This proves the statement.
\end{proof}

Concatenating the bijection in Proposition~\ref{prop: bij rigid} with the bijections in Theorems \ref{Thm: cosilt bij} and \ref{thm:rigidCosiltBij}, we obtain the next corollary.

\begin{corollary}\label{cor: cosilting mods and pairs}
Let $A$ be a left artinian ring.  The assignment $(\Z, \Ical) \mapsto C_{(\Z, \Ical)} := \prod_{Z\in\Z}Z$ yields a bijection between $\Cosiltpair{A}$ and the set of equivalence classes of cosilting $A$-modules. 
\end{corollary}

Over a finite dimensional algebra we can now identify the cosilting pairs which correspond to functorially finite torsion pairs under the bijection in Theorem \ref{thm:pairModBijection}. 
\begin{proposition} \label{rem: small cosilt pair}
Let $\Lambda$ be a finite-dimensional algebra. 

(1) A module $C\in\mod{\Lambda}$  is  $\tau^{-}-$rigid if and only if it has property (ii) in Definition~\ref{rpair}, and it is support $\tau^{-}$-tilting if and only if it is a cosilting module.

(2) In a cosilting pair $(\Z,\mathcal{I})$,  the set $\Z$    is a finite set of finitely generated indecomposable modules if and only if $C_{(\Z, \Ical)}$ is a support $\tau^{-}-$tilting module.

(3) Under the bijection in Theorem~\ref{thm:pairModBijection}, the cosilting pairs in (2) correspond to the functorially finite torsion pairs in $\tors{\Lambda}$.
\end{proposition}
 \begin{proof}
 (1) is the dual version of \cite[Proposition 3.15]{AngeleriMarksVitoria:16}.
  We give an argument for the reader's convenience. By definition a module $C\in\mod{\Lambda}$  is $\tau^{-}-$rigid if and only if for any two indecomposable summands $X,Y$ of $C$ we have that $\Hom{\Lambda}( \tau^{-}Y,  X) = 0 $, which means precisely that $ \Ext{\Lambda}^1(\mathrm{Sub}\,X, Y) = 0$, where $\mathrm{Sub}\,X$ is the class of all submodules of finite products  of copies of $X$, see  \cite[Proposition 5.6]{AuslanderSmalo:81}.  By Remark~\ref{Rem: rigid C sigma}, the latter condition amounts to property (ii) in Definition~\ref{rpair}. 
  
 (2) follows immediately from (1) or from the dual version of \cite[Proposition 2.3]{AdachiIyamaReiten:14}.

 (3) Under the bijection in Theorem~\ref{thm:pairModBijection}, a cosilting pair $(\Z,\mathcal{I})$ as in (2) is taken to the torsion pair $({}^{\perp_0}C\cap\mod{\Lambda}, \mathrm{Sub}\,C)$ where $C=C_{(\Z, \Ical)}$ is the associated support $\tau^{-}-$tilting module.
 The statement thus follows from  \cite[Theorem 2.15]{AdachiIyamaReiten:14}.
\end{proof}

The bijective assignment in Theorem \ref{Thm: cosilt bij} taking cosilting modules to cosilting complexes is not constructive. It is therefore not obvious how to explicitly write down the inverse $C \mapsto (\Z_C, \Ical_C)$ to the bijection in Corollary~\ref{cor: cosilting mods and pairs}.  We spend the remainder of this section addressing this question.

\begin{proposition}\label{prop: cosilt pair prod cosilt mod}
Let $A$ be a left artinian ring $A$ and let $C$ be a cosilting $A$-module with associated cosilting torsion pair  $ (\T, \F)$ in $\Mod{A}$. If $(\Z_C, \Ical_C)$ is the cosilting pair which is mapped to $C$ in Corollary \ref{cor: cosilting mods and pairs}, then 
$\Z_C=\Prod{C} \cap \Zg{A}=\F \cap \F^{\perp_1} \cap \Zg{A}$.

In particular, the set  $\Z$ in a cosilting pair $(\Z,\mathcal I)$ is always a closed set in $\Zg{A}$.

\end{proposition}
\begin{proof}
The bijection in Theorem \ref{Thm: cosilt bij} associates $C$ with a cosilting complex $\sigma$ in $\Der{A}$ which is an injective copresentation of $C$. Theorem \ref{thm:rigidCosiltBij} maps $\sigma$ to the maximal rigid set $\N_\sigma = \Prod{\sigma}\cap \Zg{\Der{A}}$.  Now, concatenating these bijections  with the assignment in Proposition \ref{prop: bij rigid}, we get that $\Z_C = \H{}(\N_\sigma)$.
Let us verify that $\H{}(\N_\sigma)=\Prod C\cap\Zg{A}$.  If $\mu\in\N_\sigma$, then  $\H{}(\mu)\in\Prod C\cap\Zg{A}$ by construction and Lemma~\ref{Lem: homologypi}. Conversely, if $N\in\Prod C\cap\Zg{A}$, then the 2-term complex $\mu_N$ corresponding to its injective copresentation lies in $\Prod{\sigma}$, and it is indecomposable by Lemma~\ref{lem: 2-term Zg to modules Zg}, hence $\mu_N\in\N_\sigma$ and $N\in\H{}(\N_\sigma)$. 
This proves the first equality. For the second equality, recall from \cite[Rem.~3.4]{Angeleri:18} that $\Prod{C} = \F \cap \F^{\perp_1}$.  
The last statement follows from Corollaries~\ref{cor: cosilting mods and pairs} and~\ref{lem:cosiltElemCog}.
\end{proof}

To identify the set $\Ical_C$, we need to specialise the ring to an artin algebra $\Lambda$.  

\begin{lemma}\label{lem: artin alg injectives}
Let $\Lambda$ be an artin algebra and $\sigma$  a cosilting complex with $C: =\H{}(\sigma)$.  Then the following statements hold for the set $\J_C := \InjInd{\Lambda} \cap C^{\perp_0}$.
\begin{enumerate}
\item $\Cogen{C} \subseteq {}^{\perp_0}\J_C$.
\item The cosilting complex $\sigma$ is equivalent to the cosilting complex $ \omega = \mu_C \oplus I_C[-1]$ where  $ I_C $ is the direct sum of the modules in $\mathcal{J}_C$.
\end{enumerate}
\end{lemma}
\begin{proof}
(1) The set ${}^{\perp_0}\J_C$ is clearly closed under submodules and so it suffices to show it is closed under products. We observe that $\InjInd{\Lambda}$ is a finite set of (equivalence classes of) finitely presented modules and so, for any $\Ical \subset \InjInd{\Lambda}$, the direct sum $I$ of representatives of the elements of $\Ical$ is a finitely presented module.  Thus the subcategory $ {}^{\perp_0}I$ is definable \cite[Example 2.3]{Dean:19} and hence closed under products. 

(2) By assumption and Theorem~\ref{Thm: cosilt bij} we know that $C$ is a cosilting module with $\Cogen{C}=\C_\sigma$. By Lemma~\ref{Lem: complexes determined by modules} we can decompose $ \sigma$ as a 2-term complex into $\mu_C \oplus \nu \oplus I[1] $ where $ \nu $ is an isomorphism and $ I $ an injective module. Since $\mathcal{C}_\nu = \Mod{\Lambda}$ and $\mathcal{C}_\sigma = \mathcal{C}_{\mu_C} \cap \mathcal{C}_\nu \cap {}^{\perp_0} I$, we may assume that $ \nu = 0 $.
Moreover, since $ \Cogen{C} \subseteq {}^{\perp_0} I $, every indecomposable summand of $ I $ must occur as a summand of $ I_C $.  
Thus $ {}^{\perp_0}I_C \subseteq {}^{\perp_0}I $ and $ \mathcal{C}_\omega \subseteq \Cogen{C} $.  
Conversely, $ \Cogen{C} \subseteq {}^{\perp_0}I_C$ by (1), and  $ \Cogen{C} = \mathcal{C}_{\sigma} \subseteq \mathcal{C}_{\mu_C} $. So, we conclude that $ \mathcal{C}_\omega = \Cogen{C} $. This shows that $C$ is also cosilting with respect to  $\omega$. Thus $\sigma$ and $\omega$ are equivalent cosilting complexes by Theorem~\ref{Thm: cosilt bij}.
\end{proof}

\begin{theorem}\label{thm: cosilting modules and pairs}
Let $\Lambda$ be an artin algebra.  The assignment $(\Z, \Ical) \mapsto C_{(\Z, \Ical)} := \prod_{Z\in\Z}Z$ yields a bijection between $\Cosiltpair{\Lambda}$ and the set of equivalence classes of cosilting $\Lambda$-modules.  The inverse assignment is given by $C \mapsto (\Z_C, \Ical_C)$ where 
	\[\Z_C = \Prod{C}\cap\Zg{\Lambda} \text{ and } \Ical_C = \InjInd{\Lambda} \cap C^{\perp_0}.\]
Moreover, if $(\T,\F)$ is the cosilting torsion pair cogenerated by $C$, then	
\[\Z_C = \F\cap\F^{\perp_1}\cap\Zg{\Lambda}  \text{ and } \Ical_C = \InjInd{\Lambda} \cap \F^{\perp_0}.\]
\end{theorem}
\begin{proof}
The first part of the statement is given in Corollary \ref{cor: cosilting mods and pairs}.  We prove the inverse assignment is as in the statement.  Let $C$ be a cosilting $\Lambda$-module and let $(\Z_C, \Ical_C)$ be the cosilting pair determined by $C$ under the bijection in Corollary \ref{cor: cosilting mods and pairs}.
For the shape of $\Z_C$ we refer to
 Proposition \ref{prop: cosilt pair prod cosilt mod}. We verify the statements on $\Ical_C$.

Without loss of generality, we can assume that $C = \prod_{N\in\Z_C} N$.  By Proposition~\ref{prop: bij rigid} and Theorem~\ref{thm:rigidCosiltBij}, we have that 
	\[\sigma_C := \prod_{N\in\Z_C} \mu_N\oplus\prod_{I\in \Ical_C}I[-1]\] 
is a cosilting complex with $\H{}(\sigma_C)= C$.  By Lemma~\ref{lem: artin alg injectives}(2) we have that 
$\sigma_C$ is equivalent to the cosilting complex $\omega=\mu_C\oplus I_C[-1]$ where $I_C$ is the direct sum of the modules in $\mathcal{J}_C:= \InjInd{\Lambda} \cap C^{\perp_0}$.  Applying Proposition~\ref{prop: bij rigid}  and Theorem \ref{thm:rigidCosiltBij} again, we have that 
	\[\Ical_C = \{I\in\InjInd{\Lambda}\mid I[-1]\in\N_\omega\}\]
where $\N_\omega =\Prod{\omega}\cap\Zg{\Der{\Lambda}}$.  Clearly every complex $I[-1]$ with $I\in\mathcal{J}_C$ belongs to $\N_\omega$. Conversely, if $I[-1]$ lies in $\N_\omega$, then  $\Hom{\Der{A}}(\mu_C, I) \cong \Hom{\Der{A}}(\mu_C, (I[-1])[1])  = 0$ and it follows as in Proposition \ref{prop: bij rigid} that  $I\in\mathcal{J}_C$.  We conclude that $\Ical_C = \J_C$.

Finally, by Lemma \ref{lem: artin alg injectives}(1), we have that $\Ical_C\subseteq \F^{\perp_0}\cap\InjInd{\Lambda}$.  Since the other inclusion is clear, we have proved that $\Ical_C = \F^{\perp_0}\cap\InjInd{\Lambda}$.
\end{proof}

\section{Critical modules}

In the last section we have established a one-to-one correspondence between torsion pairs in $\tors{A}$ and cosilting pairs. 
In this section we will investigate the modules that arise as elements of $\Z$ for some cosilting pair $(\Z, \Ical)$. Our focus will lie on a special class of these modules, called critical, which are intimately related to the lattice structure of $\tors{A}$. 
In the next section, we will use these modules  to find an analogue of the brick\,--\,$\tau$-rigid correspondence established in \cite{DemonetIyamaJasso:19}. 

We start out by fixing some notation.  As before,  $R$ will denote an arbitrary ring and $A$ a left artinian ring.  We will write $\brick{A}$ for the set of (isoclasses of) bricks in $\mod{A}$, i.e.~of finitely generated modules whose endomorphism ring is a skew-field. Moreover, following the notation in \cite{Sentieri:22}, we will denote the smallest torsionfree class in $\mod{A}$ containing a module $M$ by $\torsf{M}$ (this torsionfree class exists by Proposition \ref{prop: closure properties}).

Next, we establish the appropriate generalisation of $\tau^-$-rigid modules.

\begin{proposition}\label{prop: strongly rigid}
Let $A$ be an artinian ring. The following statements are equivalent for a module $N\in\Zg{A}$.
\begin{enumerate}
\item $N$ is contained in $\Z$ for some cosilting pair $(\Z, \Ical)$.
\item $\mu_N$ is rigid in $\Der{A}$ where $\mu_N$ is as in Notation \ref{Not: minimal inj cop}.
\item $N$ has an injective copresentation $\sigma$ such that $\C_\sigma$ is a torsionfree class and $N\in\C_\sigma$.
\end{enumerate}
\end{proposition}
A module satisfying the above equivalent conditions will be called a \textbf{\srigid}.
\begin{proof}
(1)$\Rightarrow$(2): This follows by the definition of cosilting pair.

(2)$\Rightarrow$(3) Take $\sigma =\mu_N$.  Then $\C_\sigma$ is a torsionfree class by Lemma \ref{cor:cSigmaCosilting} and $N\in\C_\sigma$ by Remark \ref{Rem: rigid C sigma}(1).

(3)$\Rightarrow$(1) By Lemma \ref{lem: completion} there exists a cosilting complex $\sigma\oplus\rho$.  By Theorems \ref{thm:rigidCosiltBij} and \ref{thm:pairModBijection}, there exists a cosilting pair $(\H{}(\N_{\sigma\oplus\rho}), \Ical)$ where $\N_{\sigma\oplus\rho} = \Prod{\sigma\oplus\rho}\cap \Zg{\Der{A}}$.  Since $\mu_N$ is an indecomposable summand of $\sigma$, we have that $N \in \H{}(\N_{\sigma\oplus\rho})$.
\end{proof}

Note that $\Cogen{N}\subseteq{}^{\perp_1}N$ for any \srigid\  $N$ by \cite[Lem.~3.3]{Angeleri:18}, and so the pair $({}^{\perp_0}N, \Cogen{N})$ is a (not necessarily cosilting) torsion pair by the dual of \cite[Lem.~2.3]{AngeleriMarksVitoria:16}.

\begin{notation}
Let $N$ be a \srigid\ in $\Mod{A}$. We denote \[(\t_N, \f_N):=({}^{\perp_0}N\cap\mod{A}, \Cogen{N}\cap\mod{A})\] the torsion pair induced by $N$ in $\tors{A}$. It determines the following mathematical objects:
\begin{itemize}
\item a cosilting torsion pair $(\T_N, \F_N) := (\varinjlim \t_N, \varinjlim \f_N)\in\Cosilt{A}$ by Theorem \ref{Thm: cosilt bij};
\item a cosilting pair $(\Z_N, \Ical_N)$ where $\Z_N := \F_N \cap \F_N^{\perp_1} \cap \Zg{A}$ by  Proposition \ref{prop: cosilt pair prod cosilt mod}, while $\Ical_N$ is a set of indecomposable injective modules, and $\Ical_N=\InjInd{A} \cap \F_N^{\perp_0}$ when $A$ is an artin algebra by Theorem \ref{thm: cosilting modules and pairs};
\item a cosilting module $C_N := \prod_{Z \in \Z_N}Z$ where $\Z_N$ is as above, by Corollary \ref{cor: cosilting mods and pairs}.
\end{itemize}
We remark that $\F_N$ is the smallest cosilting torsionfree class containing $N$.
\end{notation}

 \begin{lemma}\label{lem: N is tf}
For every \srigid\  $N$ in $\Mod{A}$, we have that $N\in \Z_N$.\end{lemma}
\begin{proof}
We know from  Proposition~\ref{prop: cosilt pair prod cosilt mod} that $\Z_N=\F_N \cap \F_N^{\perp_1} \cap \Zg{A}$, so we need to show that $N$ is Ext-injective in $\F_N$. Since $N$ is the direct limit of its finite-dimensional subobjects $\{N_i\}$, which are all contained in $\f_N$, we have that $N$ belongs to $\F_N$. Moreover, if we pick $M\in\F_N$ and write it as $M=\varinjlim M_i$ with $M_i\in\f_N$,   then $\Ext{A}^1(M,N)\cong \Ext{A}^1(\varinjlim M_i, N) \cong \varprojlim\Ext{A}^1(M_i, N) = 0$ because $N$ is pure-injective (cf.~\cite[Lem.~6.28]{GoebelTrlifaj:12}).
\end{proof}

Our next aim is to relate  \srigid s with certain bricks. To this end 
we need the notions of torsion-free, almost torsion and critical modules, first defined in \cite{AHL}.
 
\begin{definition}\label{almost}
Let $\A$ be an abelian category and let $t = (\T, \F)$ be a torsion pair in $\A$.  We say that $ B $ is \textbf{torsion-free almost torsion} with respect to $t$, if:
\begin{enumerate}
\item $ B \in \mathcal{F} $, but every proper quotient of $ B $ is contained in $ \mathcal{ T } $. 
\item For every short exact sequence $ 0 \to B \to F \to M \to 0 $, if $ F \in \mathcal{F} $, then $ M \in \mathcal{F} $.
\end{enumerate} 
\end{definition}

Dually, one defines \textbf{torsion, almost torsion-free} modules.
When $\A=\mod{\Lambda}$ for a finite-dimensional algebra $\Lambda$, these modules coincide with  the  minimal extending and minimal coextending modules  considered in \cite{BarnardCarrollZhu:19}. 

\begin{proposition}\label{simples in the heart} Let $A$ be a left artinian ring and consider a torsion pair $t = (\T, \F)$ in 
$\Cosilt{A}$, and  $(\t, \f)$ the torsion pair in $\mod{A}$ associated to $t$ in Theorem \ref{Thm: cosilt bij}.
\begin{enumerate}
\item \cite[Proposition 2.11]{Sentieri:22} 
A finitely generated $A$-module  is torsion-free almost torsion with respect to $t$ if and only if it is torsion-free almost torsion with respect to $(\t,\f)$. 
\item \cite[Theorem 3.6]{AHL} Let  $\heart{}$ be the heart of the HRS-tilt $\mathbb{\ststr}_{t^-}$ of $t$ in $\Der{A}$.
The finitely presented simple objects in  $\heart{}$  are precisely the objects of the form $S = F$ or $S = T [-1]$   where $F\in\mod{A}$ is torsion-free, almost torsion, or $T\in\mod{A}$ is torsion, almost torsion-free with respect to $t$. 
\end{enumerate}
\end{proposition}

It follows from Proposition~\ref{simples in the heart}(2) that the torsionfree, almost torsion modules with respect to a given torsion pair form a set of pairwise Hom-orthogonal bricks, that is, a semibrick in the sense of \cite{Asai:20}.  In particular, every torsionfree, almost torsion module is a brick. By a dual of the proof given in \cite[Lem.~3.11]{Sentieri:22} (restricted to $\mod{A}$), we also have a converse statement summarised in the following lemma. 

\begin{lemma}\label{lem: brick unique tf/t}
Let $B\in \brick{A}$.  Then $B$ is the unique torsionfree, almost torsion in $\torsf{B}$. \end{lemma}

We now turn to the notion of a critical module.  Critical modules in a general definable subcategory were defined in \cite{AHL} and, in the case where the definable subcategory is a torsionfree class, the next definition is equivalent to the one given there by \cite[Prop.~5.14]{AHL}.

\begin{definition}
Let $ \mathcal{C} \subseteq \Mod{R} $ be an additive subcategory. A morphism $ f : C \to C' $ in $ \mathcal{C} $ is \textbf{left almost split in $ \mathcal{C} $}  if it is not a split monomorphism and for every morphism $ g : C \to C'' $ which is not a split monomorphism there exists $ h : C' \to C'' $ such that $ g = h \circ f $. 
\end{definition}

\begin{definition}
Let $(\T, \F)\in\Cosilt{R}$ be a cosilting torsion pair.  We say that $N\in\F$ is \textbf{critical} in $\F$ if there exists a left almost split map $N\to{N'}$ in $\F$ that is an epimorphism.
\end{definition}

We wish to make use of the results in \cite{AHL}  which describe the connection between
 torsionfree, almost torsion modules and critical modules with respect to a torsion pair $(\T, \F)\in\Cosilt{A}$ in the case where $(\T, \F) = ({}^{\perp_0}C, \Cogen{C})$ for a cotilting module $C$. 
In order to upgrade the results  from cotilting to cosilting, we will use
that every cosilting module is a cotilting module over the factor ring $\bar{R} = R/\Ann{C}$, cf.~Proposition~\ref{prop: cosilt is local cotilt}. Identifying $\Mod{\bar{R}}$ with the corresponding full subcategory of $\Mod{R}$, we note that $\Cogen{C}$ is contained $\Mod{\bar{R}}$ and $({}^{\perp_0}C\cap\Mod{\bar{R}}, \Cogen{C})$ is the cotilting torsion pair associated to $C$ in $\Mod{\bar{R}}$.  See \cite{Angeleri:18} for more details.

\begin{notation}\label{fixC} 
For any cosilting torsion pair $(\T, \F) \in \Cosilt{R}$ with associated cosilting module $C$, set $ \bar{R} = R/\Ann{C}$, and denote by  $(\overline{\T}, \overline{\F}) := (\mathcal{T} \cap \Mod{\bar{R}},\, \mathcal{F}) $ the cotilting torsion pair in $ \Mod{\bar{R}}$.
\end{notation}

We summarise the results we will need from \cite{AHL} in the next proposition.

\begin{proposition}\label{prop: upgrade tf/t critical}
The following statements hold for a cosilting torsion pair $t=(\T, \F)$ with associated cosilting module $C$.
\begin{enumerate}
\item An $R$-module $ B $ is torsion-free, almost torsion with respect to $ (\T, \F)  $ if and only if it is torsion-free, almost torsion with respect to $  (\overline{\T}, \overline{\F})  $.
\item Consider a short exact sequence
\[
\begin{tikzcd}
0 \arrow[r] & F \arrow[r, "g"] & M \arrow[r, "a"] & {M'} \arrow[r] & 0
\end{tikzcd}
\] in $\Mod{R}$.  We have that $ F $ is torsion-free, almost torsion with respect to $ t $ and $ g $ is a $ \Prod{C}-$envelope if and only if $M$ is critical in $\F$ and $a$ is a left almost split map in $\F$.
\end{enumerate}
\end{proposition}
\begin{proof}
(1)  The statement is immediate, as all the conditions which must be checked can be verified in the torsion-free class $\F$, which is contained in $ \Mod{\bar{R}} $.

(2) The statement is proved in \cite[Thm.~4.2, Cor.~5.18]{AHL} for cotilting modules, noting that the special $\F^{\perp_1}$-envelope in $\Mod{\overline{R}}$ coincides with a $\Prod{C}$-envelope in this case. It can be extended to the cosilting case by using (1). 
We begin by observing that an $R$-module $ M $ is critical with respect to $t$ if and only if it is critical with respect to $(\bar{\T},\bar{\F}) $ as an $ \bar{R}-$module, as the conditions only involve the torsion-free class $ \mathcal{F} =\bar{\F}$. Thus it suffices to show that such a sequence is contained in $\Mod{\overline{R}}$ under either assumption in the if-and-only-if statement.  In fact, the whole sequence is contained in $\F$ in both cases.  Indeed, if $ F $ is torsion-free, almost torsion and $ g $ is a $ \Prod{C}-$envelope, then ${M'}$ is contained in $\F$ by definition of torsionfree, almost torsion.  If $M$ is critical in $\F$ and $a$ is a left almost split map in $\F$, then $M,{M'}$ and hence $F$ are all contained in $\F$ as desired.
\end{proof}

\begin{theorem}\label{thm: upgrade}
Given  a cosilting torsion pair $t=(\T, \F)$ with associated cosilting module $C$, there is a bijection between the torsion-free, almost torsion modules $F$ and the critical modules $M$ with respect to $t$. The bijection assigns to $F$  the unique module $M\in \Z_C:=\Prod{C} \cap \Zg{R}$ such that $\Hom{R}(F, M)\neq 0$.
\end{theorem}
\begin{proof}
 Note that every torsionfree, almost torsion $F$ has a $\Prod{C}$-envelope (that coincides with the special $\F^{\perp_1}$-envelope in $\Mod{\overline{R}}$), and so a short exact sequence as in Proposition~\ref{prop: upgrade tf/t critical}(2) exists for every such $F$.  It is also shown in \cite{AHL} that every critical module is indecomposable, so $M \in \Z_C$ with $\Hom{R}(F, M)\neq 0$. Moreover, $\Hom{R}(F,X)=\Hom{\bar{R}}(F,X)=\Hom{\Der{\bar{R}}}(F, X) = 0$ for all $X\in \Z_C$ with $X \neq M$ because $M$ is the injective envelope of the simple object $F$ in the cotilting heart by \cite[Prop.~4.1]{AHL}.  Thus the assignment  is well-defined, and it is a bijection by Proposition~\ref{prop: upgrade tf/t critical}(2).
\end{proof}

\begin{corollary}\label{lem: assignment injective}
For any \srigid\  $N$ in $\Mod{A}$, there is at most one torsionfree, almost torsion module $B$ in $\f_N$, and $N$ is the critical module associated to $B$.  
\end{corollary}
\begin{proof}
Suppose $B$ is torsionfree, almost torsion in $\f_N$ and hence, by Proposition \ref{simples in the heart}(1), also in $\F_N$.  Then $\Hom{A}(B, N) \neq 0$ by definition, and $N\in\Z_N$ by Lemma~\ref{lem: N is tf}.  By Theorem \ref{thm: upgrade}, we have that $N$ is the critical module associated to $B$ and $B$ is the unique torsionfree, almost torsion in $\f_N$.
\end{proof}

\section{Brick--critical correspondence}

Let again  $A$ denote a left artinian ring and $\Lambda$ a finite-dimensional algebra.  
  We aim to 
extend the brick\,--\,$\tau$-rigid correspondence established in \cite{DemonetIyamaJasso:19}; however, for a cosilting pair $(\Z, \Ical)$ over a finite-dimensional algebra $\Lambda$, the finite-dimensional elements of $\Z$ are $\tau^-$-rigid rather than $\tau$-rigid (see Proposition~\ref{rem: small cosilt pair}).  As such, our correspondence will instead generalise the dual version of Demonet, Iyama and Jasso's correspondence: the brick\,--\,$\tau^-$-rigid correspondence, which is stated as Proposition \ref{Prop: DIJ dual} below, cf.~\cite[Thm.~2.3]{Asai:20} for the related semibrick\,--\,support $\tau^-$-tilting correspondence.

Recall that $\tors{A}$ is a complete lattice ordered by inclusion of torsion classes. Let us recall some lattice--theoretic notions we will use in the sequel.
\begin{definition}
Let $(L,\ge)$ be a lattice.
Given $x,y\in L$, we say that $x$ \textbf{covers}  $y$ if $x> y$ and for all $z\in L$ with $x\ge z>y$ we have $x=z$.
Moreover, an element $x$ in  $L$ is said to be  \textbf{completely meet irreducible} if, whenever $x = \bigwedge_{i\in I} y_i$ with $y_i \in L$, we must have $x = y_j$ for some $j \in I$.  This condition can be restated as follows: there is a unique element $x^*$ covering $x$, and for every $y>x$ we have $y \geq x^*$.  
\end{definition}

 The following result was shown independently in \cite{BarnardCarrollZhu:19} and \cite{DemonetIyamaReadingReitenThomas:23}  for a finite-dimensional algebra $\Lambda$, the same arguments work over the left artinian ring $A$.
 \begin{proposition}\label{prop: cmi}
 There is a bijection between $\brick{A}$ and the completely meet irreducible elements of the lattice $\tors{A}$ which associates to $B\in\brick{A}$  
the
 torsion pair $({}^{\perp_0}B\,\cap\,\mod{A}, \torsf{B})$.\end{proposition}

Let $\Ri(A)$ denote the set of isomorphism classes of  \srigid s $N$ such that the torsion class $\t_N$ is completely meet irreducible, and for every $N\in \Ri(A)$, let $B_N$ denote the unique brick such that $\f_N = \torsf{B_N}$.

The main result of this section is the following correspondence.

\begin{theorem}\label{Thm: brick-critical correspondence}
Let 
$A$ be a left artinian ring. There is a bijection 
	\[\Ri(A) \to \brick{A}\]
given by $N\mapsto B_N$ where $B_N$ denotes the unique brick such that $\f_N = \torsf{B_N}$.
\end{theorem}
\begin{proof}  
The assignment is well-defined by definition.  If $N\in \Ri(A)$, then $B_N$ is torsionfree, almost torsion in $\f_N$ by Lemma \ref{lem: brick unique tf/t} and $N$ is the associated critical module by  
 Corollary \ref{lem: assignment injective}. So the assignment is injective, and it 
 remains to show that it is surjective. Given  $B\in\brick{A}$, we have to find a module $N\in \Ri(A)$ such that $\f_N = \torsf{B}$.
Consider the cosilting torsionfree class given by $\mathcal{V}_B:=\varinjlim(\torsf{B})$ and note that $B$ is torsionfree, almost torsion with respect to the corresponding torsion pair by Lemma \ref{lem: brick unique tf/t} and Proposition \ref{simples in the heart}(1). If $N$ is the critical module in $\mathcal{V}_B$ associated to $B$, then $B$ is a submodule of $N$ by Proposition~\ref{prop: upgrade tf/t critical}(2), and so $B\in\Cogen{N}\cap\mod{A} = \f_N$.  Therefore $\torsf{B} \subseteq \f_N \subseteq \mathcal{V}_B\cap \mod{\Lambda} = \torsf{B}$ and so $\f_N = \torsf{B}$.  In other words, we have shown that $N\in\Ri(\Lambda)$ is a preimage of $B$.
\end{proof}
The final part of this section is dedicated to proving that Theorem \ref{Thm: brick-critical correspondence} does indeed extend the dual of the brick\,--\,$\tau$-rigid correspondence found in \cite{DemonetIyamaJasso:19}.  

\begin{definition}\label{Def: reject}
Let $N \in\Zg{A}$ be a \srigid\ over a left artinian ring $A$ and consider the canonical monomorphism $\Phi \colon N\to N^{\Rad(E)}$ where $\Rad(E)$ denotes the Jacobson radical of the endomorphism ring $E := \End{A}(N)$ of $N$.  Let $\tilde{N} := \Im{\Phi}$ and $S_N := \Ker{\Phi}$.  We consider the \textbf{reject sequence}
	\[0 \to S_N \overset{f}{\rightarrow} N \overset{g}{\rightarrow} \tilde{N} \to 0. \]
\end{definition}

\begin{lemma}\label{lem: envelope and las}
Let $N\in\Zg{A}$ be a \srigid.  If $S_N\neq 0$ in the reject sequence defined in Definition \ref{Def: reject}, we have that $g$ is left almost split in $\Cogen{N}$ and $f$ is an $(\F_N\cap\F_N^{\perp_1})$-envelope of $S_N$.
\end{lemma}
\begin{proof}
We begin by proving that $g$ is left almost split in $\Cogen{N}$.  By definition, both $N$ and $\tilde{N}$ belong to $\Cogen{N}$.  Now consider a morphism $h \colon N \to M$ with $M\in\Cogen{N}$ that is not a split monomorphism and consider some monomorphism $k \colon M\to N^I$.  For each $i\in I$, let $\pi_i\colon N^I \to N$ be the $i$th projection and consider the composition $\pi_ikh$.  Then $\pi_ikh$ belongs to $\Rad(E)$ because $E:=\End{A}(N)$ is local and so $\Rad(E)$ consists of all of the non-isomorphisms.  Indeed, if $\pi_ikh$ were an isomorphism, then $h$ would be a split monomorphism.  Clearly $\pi_ikhf = 0$ for all $i\in I$ and hence $ khf = 0$. Since $k$ is a monomorphism we conclude that $hf =0$ and so there exists a unique $g'$ such that $h=g'g$ as required.

Next we show that $f$ is an $(\F_N\cap\F_N^{\perp_1})$-envelope of $S_N$.  By Lemma \ref{lem: N is tf}, we have that $N$ belongs to $\F_N\cap\F_N^{\perp_1}$, and  $\tilde{N} \in \F_N$ by construction.  A standard argument yields that $f$ is therefore a  $(\F_N\cap\F_N^{\perp_1})$-preenvelope.  It remains to show minimality.  Suppose $m\in E$ such that $f = mf$.  If $m$ is a split monomorphism, then $m$ is an isomorphism because $E$ is local.  If $m$ is not a split monomorphism, then there exists some $g'$ such that $m=g'g$ because $g$ is left almost split in $\Cogen{N}$.  But then $f = mf = g'gf = 0$, which contradicts our assumption that $S_N \neq 0$.
\end{proof}

We are now ready to state and prove that our bijection restricts to the dual of \cite[Thm.~4.1]{DemonetIyamaJasso:19}. Let $\Lambda$ be a finite-dimensional algebra and let $\tau$ denote the Auslander-Reiten translate.  We fix the notation 
	\[\tau^-\!\text{-}\mathbf{rigid}(\Lambda) := \{X\in\mod{\Lambda} \mid X \text{ is indecomposable and } \Hom{\Lambda}(\tau^-X, X) = 0\}.\] 
	\begin{lemma}\label{lem: ff tau} Over a finite dimensional algebra $\Lambda$, we have that
\[\tau^-\!\text{-}\mathbf{rigid}(\Lambda) = \Ri(\Lambda) \,\cap\, \mod{\Lambda}=\{N\!\in\!\Mod{\Lambda}\mid N \text{ is a grain with $\f_N$ functorially finite} \}.\] 
	 Moreover,	
if $N\in\tau^-\!\text{-}\mathbf{rigid}(\Lambda)$, then $\F_N = \Cogen{N}$.
\end{lemma}
\begin{proof}
By Proposition \ref{rem: small cosilt pair}, we have that an indecomposable module $N\in\mod{\Lambda}$ is $\tau^-$-rigid if and only if it is a \srigid.  Then $\f_N$ coincides with the torsion-free class $\mathrm{Sub}\,N$ in $\mod{\Lambda}$ given by submodules of finite direct sums of $N$, which is functorially finite by \cite[Thm.~5.10]{AuslanderSmalo:81}. Moreover, $\t_N$ is completely meet irreducible by the dual version of \cite[Thm.~4.8]{DemonetIyamaReadingReitenThomas:23}. We include the argument for the reader's convenience. Since $\f_N \neq 0$, we have that $\t_N\neq \mod{\Lambda}$.  For any $(\u', \v')\in\tors{A}$ with $\t_N \subset \u'$, we can apply \cite[Thm.~3.1]{DemonetIyamaJasso:19} to obtain a torsion pair $(\u, \v)$ with $\t_N \subset \u \subseteq \u'$ with $(\u, \v)$ covering $(\t_N, \f_N)$ in $\tors{\Lambda}$.  By \cite{BarnardCarrollZhu:19}, the torsion pairs covering $(\t_N, \f_N)$ are parametrised by the torsionfree, almost torsion modules in $\f_N$.  Therefore there is a torsionfree, almost torsion module $B$ in $\f_N$ corresponding to this cover relation and it is unique by Corollary \ref{lem: assignment injective}.  This means that $(\u, \v)$ is the unique cover of $(\t_N, \f_N)$ and it is contained in any other torsion class $\u'$ with  $\t_N \subset \u'$. So we have shown that $(\t_N, \f_N)$ is completely meet-irreducible in $\tors{\Lambda}$.  That is, we have $\tau^-\!\text{-}\mathbf{rigid}(\Lambda) = \Ri(\Lambda) \,\cap\, \mod{\Lambda}$.

For the second equality, we have already seen the inclusion ``$\subseteq$''. Suppose $N$ is a \srigid\ and $\f_N$ is functorially finite.  By Proposition \ref{rem: small cosilt pair}, we have that $\Z_N$ is a finite set of finite-dimensional modules.  By Lemma \ref{lem: N is tf}, we have that $N$ is finite-dimensional and hence $\tau^-\!\text{-}${rigid}.
Moreover, since $\f_N$ is functorially finite, there is a unique extension of $(\t_N, \f_N)$ to $\Mod{\Lambda}$ by \cite{Vitoria:20} and $(\T_N, \F_N)$ must therefore coincide with $({}^{\perp_0}N, \Cogen{N})$.
\end{proof}

\begin{proposition}\label{Prop: DIJ dual}
Let $\Lambda$ be a finite-dimensional algebra.  Let $\mathsf{f}^-\!\text{-}\mathbf{brick}(\Lambda)$ denote the set of bricks $B$ with $\torsf{B}$ functorially finite.  There is a bijection 
	\[ \tau^-\!\text{-}\mathbf{rigid}(\Lambda) \to \mathsf{f}^-\!\text{-}\mathbf{brick}(\Lambda)\]
given by $N \mapsto S_N$ where $S_N$ is the kernel of the reject sequence  in Definition \ref{Def: reject} and coincides with the socle $\Soc_EN$ of $N$ over its endomorphism ring $E$. Moreover, this bijection is the restriction of the bijection in Theorem \ref{Thm: brick-critical correspondence} to finite-dimensional modules.
\end{proposition}
\begin{proof}
First we show that the bijection in Theorem \ref{Thm: brick-critical correspondence} restricts to a bijection between $\tau^-\!\text{-}\mathbf{rigid}(\Lambda)$ and $\mathsf{f}^-\!\text{-}\mathbf{brick}(\Lambda)$.  By Lemma \ref{lem: ff tau}, we have that $\tau^-\!\text{-}\mathbf{rigid}(\Lambda) = \Ri(A) \cap \mod{A}$ and the image of the restriction of the bijection to $\tau^-\!\text{-}\mathbf{rigid}(\Lambda)$ is equal to $\mathsf{f}^-\!\text{-}\mathbf{brick}(\Lambda)$.

Next we show that, for $N\in \tau^-\!\text{-}\mathbf{rigid}(\Lambda)$, we have $S_N \cong B_N$.  By Lemma \ref{lem: envelope and las}, the map $g$ in the reject sequence is left almost split in $\F_N = \Cogen{N}$, and $N$ is critical in $\F_N$ by Corollary \ref{lem: assignment injective}.  Thus $S_N$ is the torsionfree, almost torsion corresponding to $N$ by Proposition \ref{prop: upgrade tf/t critical}(2), i.e., $S_N \cong B_N$.

Finally, note that the endomorphism ring $E$ of the finite-dimensional module $N$ is a semiprimary ring, which entails that $\Soc_EN$ consists of the elements of $N$ that are  annihilated by the Jacobson radical $\Rad(E)$. Thus $\Soc_EN=S_N$.
\end{proof}

We end this section by discussing when 
$B_N$ is the kernel of the reject sequence.  Indeed we do not know of any examples where this is not the case. A similar result has been obtained independently in \cite{EbrahimiNasrIsfahani:24+}.
\begin{proposition}\label{prop: reject}
Suppose $N\in\Zg{A}$ is a \srigid\  over a left artinian ring $A$.  Following the notation of Definition \ref{Def: reject}, we assume that $S_N$ is non-zero and finitely presented. Then the following statements hold.
\begin{enumerate}
\item $S_N$ is the unique torsionfree, almost torsion with respect to $(\t_N, \f_N)$.
\item $S_N$ is torsionfree, almost torsion with respect to $(\T_N, \F_N)$ and the reject sequence coincides with the short exact sequence in Proposition \ref{prop: upgrade tf/t critical}(2).
\item $N$ is the critical module in $\F_N$ associated to $S_N$ under the bijection in Theorem \ref{thm: upgrade}.
\end{enumerate}
If, in addition, 
$N$ is contained in $\Ri(A)$, then $S_N$ is isomorphic to $B_N$.
\end{proposition}
\begin{proof}
First note that the final statement follows directly from Lemma \ref{lem: brick unique tf/t} and statement (1).

(1) By definition $S_N\in \f_N$.  Suppose $q' \colon S_N \to Q'$ is a quotient of $S_N$ and consider the composition $q \colon S_N \to Q$ of $q'$ with the canonical epimorphism to the torsionfree part of $Q'$.  We will show that $q$ is either a monomorphism or zero.  In the first case, we can conclude that $q'$ is not a proper epimorphism and, in the second case, we can conclude that $Q'$ is contained in $\t_N$.  Hence we will have shown condition (1) of Definition \ref{almost} holds.

Since $Q$ is in $\f_N = \Cogen{N} \cap \mod{A}$, we have the solid arrows in the following diagram where $e$ is some monomorphism and $m\in\mathbb{N}$:
	\[\xymatrix{0 \ar[r] & S_N \ar[d]^{q} \ar[r]^{f} & N \ar@{-->}[d]^{f'} \ar[r]^{g}& \tilde{N} \ar[r] \ar@{-->}[dl]^{g'} &0\\
	0 \ar[r] & Q \ar[r]^{e}&N^m & &}\]
We note that $N\in \F_N \cap \F_N^{\perp_1}$ by Lemma \ref{lem: N is tf}.  By Lemma \ref{lem: envelope and las}, there exists $f'$ such that $eq = f'f$ because $f$ is an $(\F_N \cap \F_N^{\perp_1})$-envelope of $S_N$.  If $f'$ is a monomorphism, then so is $eq = f'f$ and hence $q$ is a monomorphism.  If $f'$ is not a monomorphism, then it is certainly not a split monomorphism.  Thus, by Lemma \ref{lem: envelope and las}, we have that there exists $g'$ such that $f' = g'g$ since $g$ is a left almost split morphism in $\F_N$.  But then $eq = f'f = g'gf = 0$, which implies that $q = 0$ because $e$ is a monomorphism.  We have shown that (1) of Definition \ref{almost} holds.

Next we show that condition (2) of Definition \ref{almost} holds.  Consider $F\in \f_N$ and a short exact sequence $0 \to S_N \to F \to M \to 0$. We aim to show that $M\in\f_N$.  Pushing out along $f$ gives us the following commutative diagram:
	\[\xymatrix{ 0 \ar[r] & S_N \ar[r]^{l} \ar[d]^{f} & F \ar[r] \ar[d] & M \ar@{=}[d] \ar[r] & 0\\
	0 \ar[r] & N  \ar[r]_{m} & L \ar[r] & M \ar[r] & 0
	}\]
Now $F \in \f_N \subseteq \Cogen{N}$ and $\tilde{N}\in\Cogen{N}$, so $L$ is contained in $\Cogen{N}$.  Consider a monomorphism $e\colon L \to N^I$ for some set $I$.  If $em$ is not a split monomorphism, then there exists $g'$ such that $em = g'g$ by Lemma \ref{lem: envelope and las}.  But then $g$ would be a monomorphism, which contradicts our assumption that $S_N$ is non-zero.  Thus $em$ is a split monomorphism and hence the bottom short exact sequence in the above diagram splits. It follows that $M$ is a direct summand of $L\in \Cogen{N}$ and hence $M \in \f_N$. We have shown that (2) of Definition \ref{almost} holds.

The uniqueness of $S_N$ and statements (2) and (3) follow from Corollary~\ref{lem: assignment injective}.
\end{proof}

\appendix
\section{The Ziegler spectrum}\label{App: Zg}
\subsection{The spectrum of a locally coherent Grothendieck category}\label{Sec: Zg loc coh}

In this section we will give an overview of the injective spectrum of a locally coherent Grothendieck category introduced in concurrent articles by Herzog \cite{Herzog:97} and Krause \cite{Krause:97}.  The  spectrum is a topological space whose points are isomorphism classes of indecomposable injective objects and whose topology is controlled by certain localising subcategories of the category.  We will keep definitions and explanations brief and refer the reader to the original articles for details.  

\begin{definition}\cite{Roos:69}
Let $\G$ be a Grothendieck abelian category.  An object $M$ in $\G$ is called \textbf{coherent} if it is finitely presented and every finitely generated subobject of $M$ is finitely presented.  The category $\G$ is called \textbf{locally coherent} if every object of $\G$ is a direct limit of coherent objects.
\end{definition}

In a locally coherent category $\G$, the full subcategory $\fp{\G}$ of finitely presented objects coincides with the full subcategory of coherent objects and, moreover, this subcategory is an exact abelian subcategory of $\G$; cf. \cite[Thm.~1.6]{Herzog:97}.

The following examples of locally coherent Grothendieck categories  all arise in this article.  The second example will be essential in Section \ref{Sec: Zg derived} and the third will be essential in Section \ref{Sec: Zg modules}.

\begin{example}\label{Ex: loc coh cats}\begin{enumerate}
\item Let $\A$ be a skeletally small additive category.  The category $\Mod{\A}$ of additive functors from $\A^\op$ to the category $\Ab$ of abelian groups is locally coherent if and only if $\A$ has pseudo-kernels.  Dually the category $\Modop{\A}$ is locally coherent if and only if $\A$ has pseudo-cokernels.  See, for example, the proof of \cite[Prop.~2.1]{Herzog:97}.

\item Let $\C$ be a skeletally small triangulated category.  Then $\Mod{\C}$ is a locally coherent Grothendieck category by (1).  

\item Let $R$ be a ring and let $\mod{R}$ be the category of finitely presented left $R$-modules.  Then the functor category $\Modop{\mod{R}}$ is a locally coherent Grothendieck category by (1).
\end{enumerate}
\end{example}

\begin{definition} Let $\G$ be a Grothendieck category.
A torsion pair $(\T, \F)$ in $\G$ is called \textbf{hereditary} if $\T$ is closed under subobjects. The subcategory $\T$ is then a \textbf{localising subcategory}, i.e.~one
 can  form the quotient category $\G/\T$, and the canonical quotient functor $q \colon \G \to \G/\T$ admits a right adjoint functor $s \colon \G/\T \to \G$, called the \textbf{section functor}. \end{definition}

In the next theorem we see that in a locally coherent Grothendieck category the hereditary torsion classes of finite type preserve the locally coherent structure.

\begin{theorem}{\cite{Herzog:97, Krause:97}}\label{Thm: serre bij}
Let $\G$ be a locally coherent Grothendieck category.  Let $\T$ be a hereditary torsion class of finite type in $\G$. The subcategory $\T$ and quotient category $\G / \T$ are both locally coherent Grothendieck categories.
\end{theorem}

\begin{notation}\label{Not: Spec}
Let $\G$ be a locally coherent Grothendieck category.  We denote by $\Sp{\G}$ the set of isomorphism classes of indecomposable injective objects in $\G$.  

For an object $C$ of $\fp{\G}$, we fix the notation 
	\[\open{C} = \{E \in \Sp{\G} \mid \Hom{\G}(C, E) \neq 0 \}.\] 

\end{notation}

We will reserve the notation $\Sp{\G}$ for when we are considering the set together with the topology described in the next theorem. At other points in the main body of article we  use the notation $\InjInd{R}$ for the set of isomorphism classes of indecomposable injective objects in a module category $\Mod{R}$ (considered without the topology).

\begin{theorem}{\cite{Herzog:97, Krause:97}}\label{Thm: fundamental Zg}
The collection of subsets of $\Sp{\G}$ of the form $\open{C}$ where $C$ is a finitely presented object of $\G$ form a basis of open sets for a topology on $\Sp{\G}$.  The set $\Sp{\G}$ endowed with this topology is called the \textbf{spectrum} of $\G$. 

Moreover, the following statements hold for the topological space $\Sp{\G}$: \begin{enumerate}
\item There is a bijection between the closed subsets $\I$ of $\Sp{\G}$ and the hereditary torsion pairs $(\T, \F)$ of finite type in $\G$ given by the assignments
	\[\I \mapsto ({}^{\perp_0}\I, \Cogen{\I}) \hspace{10mm} (\T, \F) \mapsto \F \cap \Sp{\G}. \]
	
\item Let $\I$ be a closed subset of $\Sp{\G}$ and let $\T = {}^{\perp_0}\I$ be  the corresponding hereditary torsion class of finite type.  Then the section functor $s \colon \G /\T \to \G$ induces a homeomorphism from the spectrum $\Sp{\G /\T}$ of $\G/\T$ to the closed subset $\I$ of $\Sp{\G}$ endowed with the subspace topology, with its inverse given by the quotient functor $q \colon \Sp{\G} \to \Sp{\G/\T}$.
\end{enumerate}
\end{theorem}

Since the hereditary torsion-free classes of finite type can be described by closure conditions (they are exactly those full subcategories that are closed under products, subobjects, extensions, injective envelopes and filtered colimits), we can consider the smallest such subcategory containing a given object.  That is, for any $M$ in $\G$, we take the hereditary torsion-free class of finite type
	\[ \F_M := \bigcap_{\F \in \mathscr{F}_M} \F \]
where $\mathscr{F}_M$ is the set of all hereditary torsion-free classes of finite type that contain $M$.  We are interested in injective objects $E$ that cogenerate the torsion pair $(\T_E, \F_E)$; we will refer to such an object as an \textbf{elementary cogenerator}.  The following theorem characterises such injective objects.

\begin{proposition}\label{Prop: ele cogen}
Let $\G$ be a locally coherent Grothendieck category and let $E$ be an injective object of $\G$.  Consider the torsion pair $(\T_E, \F_E)$ and denote the quotient functor by $q \colon \G \to \G/\T_E$.  The following statements are equivalent. \begin{enumerate}
\item The set $\Sp{\G} \cap \Prod{E}$ is a closed subset of $\Sp{\G}$.
\item The torsion pair $({}^{\perp_0}E, \Cogen{E})$ is of finite type.
\item The equality $(\T_E, \F_E) = ({}^{\perp_0}E, \Cogen{E})$ holds.
\item The object $qE$ is an injective cogenerator of $\G/\T_E$.
\end{enumerate}
\end{proposition}
\begin{proof}
$(1)\Leftrightarrow(2)$ follows from Theorem \ref{Thm: fundamental Zg}, and $(2) \Leftrightarrow (3)$ is clear.  $(3)\Leftrightarrow(4)$ follows from the fact that $q$ induces an equivalence of categories $q \colon \Inj{\G} \cap \F \to \Inj{\G/\T_E}$.
\end{proof}

\begin{remark}\label{prods}{\rm If $E$ is an elementary cogenerator in $\G$ and $\Z=\Sp{\G} \cap \Prod{E}$ is the associated closed subset of $\Sp\G$, then $\Prod{E}=\Prod{\Z}$.
Indeed, the bijection in Theorem~\ref{Thm: fundamental Zg}(1) maps $\Cogen E$ to the closed subset $\Z$. Composing with  the inverse bijection, we obtain $\Cogen E=\Cogen \Z$. In particular, the two classes have the same injectives.}
\end{remark}

\subsection{Purity and the Ziegler spectrum of the derived category}\label{Sec: Zg derived}

In Example \ref{Ex: loc coh cats}(2) we saw that a locally coherent Grothendieck category $\Mod{\C}$ can be associated to any small triangulated category $\C$.  Following \cite{Krause:02}, we can use the category $\Mod{\B^\c}$ to define the Ziegler spectrum of any compactly generated triangulated category $\B$ where $\B^\c$ denotes the full subcategory of compact objects in $\C$.  

In this article we consider the Ziegler spectrum of the derived category $\Der{R}$ of the category $\Mod{R}$ of left $R$-modules over a ring $R$.  Throughout this section $R$ will denote an associative ring and $\D^c$ will denote the full subcategory of compact objects in $\Der{R}$.

\begin{definitions}\label{Def: Purity Der}\begin{enumerate}
\item The functor $\y \colon \Der{R} \to \Mod{\D^\c}$ given by the assignments 
	\[X \mapsto \Hom{\Der{R}}(-, X)|_{\D^\c} \hspace{10mm} \text{ and } \hspace{10mm} f \mapsto \Hom{\Der{R}}(-, f)|_{\D^\c}\]
will be referred to as the \textbf{restricted Yoneda functor}.

\item A subcategory $\C$ of $\Der{R}$ is called \textbf{definable} if it is closed under products, pure subobjects and pure epimorphic images.  

\item  A pure-injective object $N$ in $\Der{R}$ is called an \textbf{elementary cogenerator} if $\y N$ is an elementary cogenerator in $\Mod{\D^\c}$.
\end{enumerate}
\end{definitions}

The definition of a definable subcategory given in (2) is adopted from \cite{LakingVitoria:20}, and it is equivalent to the original definition from \cite{Krause:02}. In particular, a subcategory $\C$ of $\Der{R}$ is definable provided it is determined  by the vanishing of a family of  representable functors, that is, $\C= \bigcap_{i\in I}\Ker{F_i}$ where $F_i= \Hom{\Der{R}}(X_i, -)|_{\D^\c}$ with $X_i\in {\D^\c}$ for  $i\in I$.

Krause showed in \cite{Krause:00} that the injective objects of $\Mod{\D^c}$ coincide with the objects of the form $\y X$ where $X$ is a pure-injective object of $\Der{R}$. This leads to the following notation, which is further justified by the subsequent theorem.  
\begin{notation} We denote the set of isomorphism classes of indecomposable pure-injective objects in $\Der{R}$ by $\Zg{\Der{R}}$.  \end{notation}

\begin{theorem}{\cite[Fundamental Correspondence]{Krause:02}}\label{Thm: Zg Der}
The subsets of $\Zg{\Der{R}}$ of the form $\C \cap \Zg{\Der{R}}$ where $\C$ is a definable subcategory of $\Der{R}$ satisfy the axioms for the closed sets of a topology on $\Zg{\Der{R}}$.  We refer to the set $\Zg{\Der{R}}$ endowed with this topology as the \textbf{Ziegler spectrum of $\Der{R}$}.
Moreover, the restricted Yoneda functor $\y$ induces a homeomorphism 
	\[ \y \colon \Zg{\Der{R}} \to \Sp{\Mod{\D^\c}}.\]
\end{theorem}

\subsection{Purity and the Ziegler spectrum of the module category}\label{Sec: Zg modules}

In Example \ref{Ex: loc coh cats}(3) we saw that a locally coherent Grothendieck category $\G_R := \Mod{\mod{R}^\op}$ can be associated to an associative ring $R$.  Following \cite{Herzog:97}, we can use the category $\G_R$ to define the Ziegler spectrum of $\Mod{R}$.  Throughout this section $R$ will denote a fixed associative ring.

\begin{definitions}\begin{enumerate}
\item The functor $\ten \colon \Mod{R} \to \G_R$ given by the assignments 
	\[X \mapsto (-\otimes_RX) \hspace{10mm} \text{ and } \hspace{10mm} f \mapsto (-\otimes_Rf)\]
will be referred to as the \textbf{restricted Yoneda functor}.

\item  A pure-injective module $N$ in $\Mod{R}$ is called an \textbf{elementary cogenerator} if $\ten N$ is an elementary cogenerator in $\G_R$. 
\end{enumerate}
\end{definitions}

The injective objects of $\G_R$ coincide with the objects of the form $\ten X$ where $X$ is a pure-injective module in $\Mod{R}$ (by \cite[Prop.~1.2]{GrusonJensen:81}, see also \cite[Prop.~4.1]{Herzog:97}). This leads to the following notation, which is further justified by the subsequent theorem.  
\begin{notation} We denote the set of isomorphism classes of indecomposable pure-injective objects in $\Mod{R}$ by $\Zg{R}$.  \end{notation}

\begin{theorem}{\cite{Crawley-Boevey:98}, \cite{Herzog:97}}\label{Thm: Zg Mod}
The subsets of $\Zg{R}$ of the form $\C \cap \Zg{R}$ where $\C$ is a definable subcategory of $\Mod{R}$ satisfy the axioms for the closed sets of a topology on $\Zg{R}$.  We refer to the set $\Zg{R}$ endowed with this topology as the \textbf{Ziegler spectrum of $\Mod{R}$}.
Moreover, the restricted Yoneda functor $\ten$ induces a homeomorphism 
	\[ \ten \colon \Zg{R} \to \Sp{{\G_R}}.\] In particular, a pure-injective module $N$ is an elementary cogenerator if and only if $\Prod N\cap\ \Zg{R}$ is a closed subset of $ \Zg{R}$.
\end{theorem}

\bibliographystyle{plain}
\bibliography{refs}

\begin{thebibliography}{10}

\bibitem{AdachiIyamaReiten:14}
Takahide Adachi, Osamu Iyama, and Idun Reiten.
\newblock {$\tau$}-tilting theory.
\newblock {\em Compos. Math.}, 150(3):415--452, 2014.

\bibitem{Angeleri:18}
Lidia Angeleri~H\"{u}gel.
\newblock On the abundance of silting modules.
\newblock In {\em Surveys in representation theory of algebras}, volume 716 of
  {\em Contemp. Math.}, pages 1--23. Amer. Math. Soc., [Providence], RI, [2018]
  \copyright 2018.

\bibitem{AHL}
Lidia Angeleri~H\"ugel, Ivo Herzog, and Rosanna Laking.
\newblock Simples in a cotilting heart, preprint 2022.

\bibitem{AngeleriLakingSentieri:24+}
Lidia Angeleri~H\"{u}gel, Rosanna Laking, and Francesco Sentieri.
\newblock Mutation of torsion pairs for finite dimensional algebras, 2024.

\bibitem{ALSV}
Lidia Angeleri~H\"ugel, Rosanna Laking, Jan {S}{t}ov\'{\i}{c}ek, and Jorge
  Vit\`oria.
\newblock Mutation and torsion pairs, preprint 2022.

\bibitem{AngeleriMarksVitoria:16}
Lidia Angeleri~H\"{u}gel, Frederik Marks, and Jorge Vit\'{o}ria.
\newblock Silting modules.
\newblock {\em Int. Math. Res. Not. IMRN}, 2016(4):1251--1284, 2016.

\bibitem{AngeleriMarksVitoria:17}
Lidia Angeleri~H\"{u}gel, Frederik Marks, and Jorge Vit\'{o}ria.
\newblock Torsion pairs in silting theory.
\newblock {\em Pacific J. Math.}, 291(2):257--278, 2017.

\bibitem{AngeleriSentieri:23+}
Lidia Angeleri~H\"{u}gel and Francesco Sentieri.
\newblock Wide coreflective subcategories and torsion pairs, preprint 2023.

\bibitem{Asai:20}
Sota Asai.
\newblock Semibricks.
\newblock {\em Int. Math. Res. Not. IMRN}, (16):4993--5054, 2020.

\bibitem{Asai:21}
Sota Asai.
\newblock The wall-chamber structures of the real {G}rothendieck groups.
\newblock {\em Adv. Math.}, 381:Paper No. 107615, 44, 2021.

\bibitem{AuslanderSmalo:81}
M.~Auslander and Sverre~O. Smal\o.
\newblock Almost split sequences in subcategories.
\newblock {\em J. Algebra}, 69(2):426--454, 1981.

\bibitem{BarnardCarrollZhu:19}
Emily Barnard, Andrew Carroll, and Shijie Zhu.
\newblock Minimal inclusions of torsion classes.
\newblock {\em Algebr. Comb.}, 2(5):879--901, 2019.

\bibitem{Beilinson:78}
A.~A. Beilinson.
\newblock Coherent sheaves on {${\bf P}\sp{n}$} and problems in linear algebra.
\newblock {\em Funktsional. Anal. i Prilozhen.}, 12(3):68--69, 1978.

\bibitem{BeilinsonBernsteinDeligne:82}
A.~A. Beilinson, J.~Bernstein, and P.~Deligne.
\newblock Faisceaux pervers.
\newblock In {\em Analysis and topology on singular spaces, {I} ({L}uminy,
  1981)}, volume 100 of {\em Ast\'{e}risque}, pages 5--171. Soc. Math. France,
  Paris, 1982.

\bibitem{Bondal:89}
A.~I. Bondal.
\newblock Representations of associative algebras and coherent sheaves.
\newblock {\em Izv. Akad. Nauk SSSR Ser. Mat.}, 53(1):25--44, 1989.

\bibitem{BreazPop:17}
Simion Breaz and Flaviu Pop.
\newblock Cosilting modules.
\newblock {\em Algebr. Represent. Theory}, 20(5):1305--1321, 2017.

\bibitem{Bridgeland:17}
Tom Bridgeland.
\newblock Scattering diagrams, {H}all algebras and stability conditions.
\newblock {\em Algebr. Geom.}, 4(5):523--561, 2017.

\bibitem{BrustleSmithTreffinger:19}
Thomas Br\"{u}stle, David Smith, and Hipolito Treffinger.
\newblock Wall and chamber structure for finite-dimensional algebras.
\newblock {\em Adv. Math.}, 354:106746, 31, 2019.

\bibitem{Crawley-Boevey:94}
William Crawley-Boevey.
\newblock Locally finitely presented additive categories.
\newblock {\em Comm. Algebra}, 22(5):1641--1674, 1994.

\bibitem{Crawley-Boevey:98}
William Crawley-Boevey.
\newblock Infinite-dimensional modules in the representation theory of
  finite-dimensional algebras.
\newblock In {\em Algebras and modules, {I} ({T}rondheim, 1996)}, volume~23 of
  {\em CMS Conf. Proc.}, pages 29--54. Amer. Math. Soc., Providence, RI, 1998.

\bibitem{Dean:19}
Samuel Dean.
\newblock Duality and contravariant functors in the representation theory of
  {A}rtin algebras.
\newblock {\em J. Algebra Appl.}, 18(6):1950111, 27, 2019.

\bibitem{DemonetIyamaJasso:19}
Laurent Demonet, Osamu Iyama, and Gustavo Jasso.
\newblock {$\tau$}-tilting finite algebras, bricks, and {$g$}-vectors.
\newblock {\em Int. Math. Res. Not. IMRN}, (3):852--892, 2019.

\bibitem{DemonetIyamaReadingReitenThomas:23}
Laurent Demonet, Osamu Iyama, Nathan Reading, Idun Reiten, and Hugh Thomas.
\newblock Lattice theory of torsion classes: beyond $\tau$-tilting theory.
\newblock {\em To appear: Trans. AMS}, 2023.

\bibitem{EbrahimiNasrIsfahani:24+}
Ramin Ebrahimi and Alireza Nasr-Isfahani.
\newblock Semibrick-cosilting correspondence, 2024.

\bibitem{GarkushaPrest:05}
Grigory Garkusha and Mike Prest.
\newblock Triangulated categories and the {Z}iegler spectrum.
\newblock {\em Algebr. Represent. Theory}, 8(4):499--523, 2005.

\bibitem{GoebelTrlifaj:12}
R\"{u}diger G\"{o}bel and Jan Trlifaj.
\newblock {\em Approximations and endomorphism algebras of modules. {V}olume
  1}, volume~41 of {\em De Gruyter Expositions in Mathematics}.
\newblock Walter de Gruyter GmbH \& Co. KG, Berlin, extended edition, 2012.
\newblock Approximations.

\bibitem{GrusonJensen:81}
L.~Gruson and C.~U. Jensen.
\newblock Dimensions cohomologiques reli\'{e}es aux foncteurs
  {$\varprojlim^{(i)}$}.
\newblock In {\em Paul {D}ubreil and {M}arie-{P}aule {M}alliavin {A}lgebra
  {S}eminar, 33rd {Y}ear ({P}aris, 1980)}, volume 867 of {\em Lecture Notes in
  Math.}, pages 234--294. Springer, Berlin-New York, 1981.

\bibitem{HappelReitenSmalo:96}
Dieter Happel, Idun Reiten, and Sverre~O. Smal\o.
\newblock Tilting in abelian categories and quasitilted algebras.
\newblock {\em Mem. Amer. Math. Soc.}, 120(575):viii+ 88, 1996.

\bibitem{Herzog:97}
Ivo Herzog.
\newblock The {Z}iegler spectrum of a locally coherent {G}rothendieck category.
\newblock {\em Proc. London Math. Soc. (3)}, 74(3):503--558, 1997.

\bibitem{IyamaYang:18}
Osamu Iyama and Dong Yang.
\newblock Silting reduction and {C}alabi-{Y}au reduction of triangulated
  categories.
\newblock {\em Trans. Amer. Math. Soc.}, 370(11):7861--7898, 2018.

\bibitem{King:94}
A.~D. King.
\newblock Moduli of representations of finite-dimensional algebras.
\newblock {\em Quart. J. Math. Oxford Ser. (2)}, 45(180):515--530, 1994.

\bibitem{KoenigYang:14}
Steffen Koenig and Dong Yang.
\newblock Silting objects, simple-minded collections, {$t$}-structures and
  co-{$t$}-structures for finite-dimensional algebras.
\newblock {\em Doc. Math.}, 19:403--438, 2014.

\bibitem{Krause:97}
Henning Krause.
\newblock The spectrum of a locally coherent category.
\newblock {\em J. Pure Appl. Algebra}, 114(3):259--271, 1997.

\bibitem{Krause:00}
Henning Krause.
\newblock Smashing subcategories and the telescope conjecture---an algebraic
  approach.
\newblock {\em Invent. Math.}, 139(1):99--133, 2000.

\bibitem{Krause:02}
Henning Krause.
\newblock Coherent functors in stable homotopy theory.
\newblock {\em Fund. Math.}, 173(1):33--56, 2002.

\bibitem{Laking:20}
Rosanna Laking.
\newblock Purity in compactly generated derivators and t-structures with
  {G}rothendieck hearts.
\newblock {\em Math. Z.}, 295(3-4):1615--1641, 2020.

\bibitem{LakingVitoria:20}
Rosanna Laking and Jorge Vit\'{o}ria.
\newblock Definability and approximations in triangulated categories.
\newblock {\em Pacific J. Math.}, 306(2):557--586, 2020.

\bibitem{MarksVitoria:18}
Frederik Marks and Jorge Vit\'{o}ria.
\newblock Silting and cosilting classes in derived categories.
\newblock {\em J. Algebra}, 501:526--544, 2018.

\bibitem{MarksZvonareva:23}
Frederik Marks and Alexandra Zvonareva.
\newblock Lifting and restricting t-structures.
\newblock {\em Bull. London Math. Soc.}, 2023.

\bibitem{Polishchuk:07}
A.~Polishchuk.
\newblock Constant families of {$t$}-structures on derived categories of
  coherent sheaves.
\newblock {\em Mosc. Math. J.}, 7(1):109--134, 167, 2007.

\bibitem{PsaroudakisVitoria:18}
Chrysostomos Psaroudakis and Jorge Vit\'{o}ria.
\newblock Realisation functors in tilting theory.
\newblock {\em Math. Z.}, 288(3-4):965--1028, 2018.

\bibitem{Roos:69}
Jan~Erik Roos.
\newblock Locally {N}oetherian categories and generalized strictly linearly
  compact rings. {A}pplications.
\newblock In {\em Category {T}heory, {H}omology {T}heory and their
  {A}pplications, {II} ({B}attelle {I}nstitute {C}onference, {S}eattle,
  {W}ash., 1968, {V}ol. {T}wo)}, pages 197--277. Springer, Berlin, 1969.

\bibitem{Saorin:2017}
Manuel Saor\'{\i}n.
\newblock On locally coherent hearts.
\newblock {\em Pacific J. Math.}, 287(1):199--221, 2017.

\bibitem{Sentieri:22}
Francesco Sentieri.
\newblock A brick version of a theorem of auslander.
\newblock {\em Nagoya Mathematical Journal}, page 1–19, 2022.

\bibitem{vitoria}
Jorge Vit\'{o}ria.
\newblock Quantity vs. size in representation theory.
\newblock {\em Bol. Soc. Port. Mat.}, (77):151--166, 2019.

\bibitem{Woolf:10}
Jonathan Woolf.
\newblock Stability conditions, torsion theories and tilting.
\newblock {\em J. Lond. Math. Soc. (2)}, 82(3):663--682, 2010.

\bibitem{ZhangWei:17}
Peiyu Zhang and Jiaqun Wei.
\newblock Cosilting complexes and {AIR}-cotilting modules.
\newblock {\em J. Algebra}, 491:1--31, 2017.

\end{thebibliography}

\end{document}